\documentclass[12pt]{article}
\usepackage{latexsym}
\usepackage{amssymb}
\usepackage{amscd}
\usepackage{amsmath,amsthm}
\usepackage{graphicx}
\usepackage{amsmath}
\usepackage{color}
\usepackage{amsfonts,amssymb}
\usepackage{mathrsfs}
\newcommand{\bfR}{{\bf R}}
\newcommand{\bfN}{{\bf N}}

\newcommand{\al}{\alpha}
\newcommand{\be}{\beta}

\newcommand{\ep}{\epsilon}

\newcommand{\ga}{\gamma}
\newcommand{\de}{\delta}

\renewcommand{\d}{\displaystyle}
\newcommand{\rme}{{\rm e}}

\newtheorem{theorem}{Theorem}[section]
\newtheorem{proposition}{Proposition}[section]
\newtheorem{remark}{Remark}[section]
\newtheorem{definition}{Definition}[section]
\begin{document}
\title{\bf Invariant subspaces of biconfluent Heun operators and special solutions of
Painlev\'e IV
}
\author{ Yik-Man Chiang$^{1}$, Chun-Kong Law$^{2}$ and Guo-Fu Yu
 $^{3}$}
\date{24 May 2019}
\maketitle

\begin{abstract}
We show that there is a full correspondence between the parameters space of the degenerate biconfluent Heun connection (BHC) and that of Painlev\'{e} IV that admits special solutions. The BHC degenerates when either the Stokes' data  for the irregular singularity at $\infty$ degenerates or the regular singular point at the origin becomes an apparent singularity. We show that if the BHC is written as isomonodromy family of biconfluent Heun equations (BHE), then the BHE degenerates precisely when it admits eigen-solutions of the biconfluent Heun operators, after choosing appropriate accessory parameter, 
  of specially constructed invariant subspaces of finite dimensional solution spaces spanned by parabolic cylinder functions.  We have found all eigen-solutions over this parameter space apart from three exceptional cases  after choosing the right accessory parameters. These eigen-solutions are expressed as certain finite sum of parabolic cylinder functions.  We extend the above sum to new convergent series expansion in terms of parabolic cylinder functions to the BHE. The infinite sum solutions of the BHE terminates precisely when the parameters of the BHE assumes the same values as those of the degenerate biconfluent Heun connection except at three instances after choosing the right accessory parameter.
\end{abstract}
Keywords: Biconfluent Heun connection/equation; parabolic cylinder functions; invariant subspaces; Painlev\'{e} IV\\
Mathematics Subject Classification (2010). 33E10, 34M35 (primary), 33E17 (secondary)
\vskip-.3cm

 \footnote{$^{1}$Department of Mathematics, The Hong Kong University of
Science and Technology, Hong Kong SAR. Email: machiang@ust.hk}
 \footnote{$^{2}$Department of Applied Mathematics,
 National Sun Yat-sen University,  Kaohsiung 804, Taiwan. Email :
 law@math.nsysu.edu.tw}
 \footnote{$^{3}$Department of Mathematics, Shanghai Jiao-Tong University,
Shanghai 200240, P.R.\ China. Email: gfyu@sjtu.edu.cn}


\section{Introduction}
\setcounter{equation}{0}
 The canonical Biconfluent Heun equation ({BHE}) (\cite{DDMRR}, \cite{maroni_a}, \cite{maroni_b}, \cite{ronv}) is written as
 \begin{equation}\label{E:BHEC}
 	x y''+(1+\alpha-\beta x -2x^2)y'+[(\gamma-\alpha-2)x-\frac{1}{2}(\delta+(1+\alpha)\beta)]y=0.
\end{equation}
where $\alpha,\, \beta,\, \gamma,\, \delta$ are parameters.
The equation is characterized by having one regular point at $x=0$ and one irregular
 singular point of rank two at $x=\infty$ which is a result of coalesce of three regular singular points of the (Fuchsian-type) Heun equation at $x=\infty$ \cite[p. 61]{wasow}.

{A standard Frobenius argument shows that}
 $$ \lim_{x\to 0} y \sim
 x^{k},
 $$
where the indices $k$ take the value $k=0, 1-c$. This illustrates that the local monodromy representation at $x=0$ is given by
	\[
	  \begin{pmatrix}
	     	0 & 0 \\
	     	0 & 1-c
	  \end{pmatrix}.
	\]

The \eqref{E:BHEC} is also known as the \textit{rotating harmonic oscillator} (e.g., \cite{nieto_gutschick}),  appeared in the second paper of the series of fundamental work \cite[\S4, (46)]{schrodinger} on classical quantum mechanics by Schr\"odinger in 1926${}^\dagger$\footnote{${}^\dagger$Schr\"odinger gave an approximation on the eigenvalues (i.e., the accessory parameters) of a special case of the BHE.}. Despite the long history of {BHE} and its frequent encounters in different branches of mathematical physics (e.g., \cite{bender}, \cite{hinterleitner}, \cite{masson}, \cite{nieto_gutschick}, \cite{petersen}),
 relatively little is known about its solutions  \cite{ronv} and the accessory parameter $\delta$ \cite{CY}. The main obstacle to better understanding the {BHE} appears that its being \textit{non-rigid} \cite{Arinkin} in the generic consideration. With the identification $\beta=2t$, and
	\[
		2\theta_0=1+\alpha,
		\quad
		2\theta_\infty=1+\gamma,
		\quad
		2(\theta_\infty-\theta_0)=\gamma-\alpha,
	\]
one can derive the BHE \eqref{E:BHEC}, via the well-known formula \eqref{E:conversion},
	\begin{equation}\label{E:BHE_JM}
			xy^{\prime\prime}+(2\theta_0-2tx-2x^2)y^\prime+ \big(2(\theta_\infty-\theta_0-1)x+4\theta_0(\lambda-t)\big)y=0,
\end{equation}
 from the Biconfluent Heun-type connection (BHC) (see Definition \ref{D:BHC-criteria}) over the rank two trivial vector bundle on the Riemann sphere $\mathbb{CP}^1$ with punctures at $z=0,\, \infty$,
  \begin{equation}\label{E:BHC}
  		\frac{d\Psi}{dx}= \Big({A}\,{x}+B+\frac{C}{x}\Big)\Psi
  		:=\mathcal{A}\,\Psi,
  	\end{equation}
where the matrices $A,\, B,\, C$ are normalised by Jimbo and Miwa in \cite[Appendix C]{JM_1981b}
	\begin{equation}\label{E:BHC-matrices}
		A=
		\begin{pmatrix}
			1 & 0 \\
			0 & -1
		\end{pmatrix},\quad
		B=
		\begin{pmatrix}
			t & u \\
			2(z-\theta_0-\theta_\infty)/u & -t
		\end{pmatrix}
		,\quad
		C=
		\begin{pmatrix}
		-z+\theta_0 & -uy/2 \\
		2z(z-2\theta_0)/uy & z-\theta_0
		\end{pmatrix},
	\end{equation}
where the matrix $C$ has eigenvalues $\pm\theta_0$ and so the local monodromy of the connection around $x=0$ or $x=\infty$, up to a conjugacy class, is given by
	\[
		T_0^{(0)}:=
		\begin{pmatrix}
			\theta_0 & 0\\
			0 & -\theta_0
		\end{pmatrix}.
	\]
In a different connection, Garnier showed \cite{garnier}, similar to Fuchs' argument \cite{fuchs} of Heun's equation and Painlev\'e VI,  that one could obtain $\mathrm{P_{IV}}$ equation via isomonodromy deformation from the {BHE}. Schlesinger \cite[(1912)]{Schlesinger} extended earlier works to differential equations in system forms with arbitrary number of regular singular points. D. V.  Chudnovsky, and  G. V. Chudnovsky \cite{chudnovsky}, and independently Jimbo, Miwa and Ueno \cite{JM_1980,JM_1981b, JMU_1981}, amongst other things, extended Schlesinger's work to differential equations with irregular singular points.

Let
	\begin{equation}\label{E:2nd-eqn}			\mathcal{B}:=Ax +B-tA=A(x-t)+B.
	\end{equation}
Then the compatibility (integrability) condition for isomonodromy deformation of \eqref{E:BHC}
	\begin{equation}\label{E:compatibility}
		\Omega=\mathcal{A}\,dx+\mathcal{B}\,dt,
		\qquad
		d\Omega=\Omega\wedge \Omega,
	\end{equation}
gives rise to Painlev\'e IV:
\begin{equation}\label{E:P4}
		\frac{d^2y}{dt^2}=\frac{1}{2y}\Big(\frac{dy}{dt}\Big)^2+\frac32y^3+4ty^2+2\big(t^2-\xi\big)y+\frac{\eta}{y}.
	\end{equation}
	
Indeed, one can derive the BHE \eqref{E:BHE_JM} as a member of the isomonodromic
deformation \eqref{E:BHC} with \eqref{E:BHC-matrices} and \eqref{E:compatibility}. To do this, one blows up the $zy-$plane at
the origin and \eqref{E:BHE_JM} appears to be the member at the exceptional point
\[
z=0,\quad y=0,\quad z/y=\lambda.
\]	
	
An important discovery by Okamoto \cite{okamoto} on Painlev\'e IV
is that the $\mathrm{P}_\mathrm{IV}$ admits special function solutions that can be written in terms of parabolic cylinder functions when	
	\begin{equation}\label{E:degeneration-set}
		\eta=-2(2n+1+\varepsilon \xi)^2,
		\quad
		\mathrm{and/or}
		\quad
		\eta=-2n^2,
		\quad
		n\in\mathbb{Z},
	\end{equation}
where $\varepsilon=\pm 1$.  The equations \eqref{E:degeneration-set} become more transparent${}^\dagger$\footnote{${}^\dagger$ The authors are unable to find a suitable reference for the \eqref{E:degeneration-set_JM}.}
\begin{equation}\label{E:degeneration-set_JM}
		\theta_0\pm \theta_\infty\in \mathbb{Z}\qquad\mathrm{and/or}\qquad 2\,\theta_0\in \mathbb{Z},
	\end{equation}
when written in terms of Jimbo-Miwa's convention \cite{JM_1981b}:
\begin{equation}\label{E:matching-coeff}
		\xi=2\,\theta_\infty-1,\qquad \eta=-8\,\theta_0^2.
	\end{equation}


We mention that if the condition ``and" in \eqref{E:degeneration-set} holds, then each parabolic  cylinder function in the corresponding special solutions further reduces to a Hermite polynomial and so these special solutions of $P_\mathrm{IV}$ are rational functions  written in terms of Hermite polynomials. Okamoto also found that the above set of special parameters are connected to the affine Weyl group of the type $\tilde{A}_2^{(1)}$ (see \cite{noumi2004}) which acts as the symmetry group of $P_\mathrm{IV}$ by way of B\"acklund transformations  \cite{noumi2004, noumi-yamada-1998, noumi-yamada-1999, okamoto}. We would like to point out that the special solutions written in terms of parabolic cylinder functions above to lie in the Picard-Viessot extension of the parabolic differential operator $L=\partial^2+(x^2+\alpha)$ for an appropriately chosen $\alpha\not=0$. Okamoto found another set of special rational solutions for $P_\mathrm{IV}$ when the corresponding $\xi,\, \eta$ satisfies
	\[	
		\xi=m,\quad
		\eta=-2\big(2n-m+\frac13)^2\big),
	\]
respectively, $\theta_0,\, \theta_\infty$ satisfy
	\[	
		\theta_0\pm \theta_\infty=\mp \big(n\pm\frac12+\frac16\big)
		=\begin{cases}
			-\big(n+\displaystyle\frac23\big)\\
			n-\displaystyle\frac13,
		\end{cases}
		\quad n\in \mathbb{Z}
	\]
or $3(\theta_0\pm \theta_\infty)\equiv 1,\, 2\  \mod 3$. This arithmetic relations amongst the $\xi,\, \eta$ (resp.   $\theta_0,\, \theta_\infty$) represent monodromy/Stokes multipliers that is \textit{incompatible} with those listed in \eqref{E:degeneration-set} and  hence the $\tilde{A}_2^{(1)}$,  so they fall outside the scope of consideration of this paper.

This paper aims to illustrate the following objectives:
	\begin{enumerate}
		\item[(i)] The monodromy/Stokes multipliers of the BHC \eqref{E:BHC} degenerate either when the differential Galois group of the BHC becomes solvable or the regular singular point at the origin becomes an apparent singularity, i.e., the monodromy at the origin becomes trivial when the $\theta_0,\, \theta_\infty$ satisfy exactly the criteria \eqref{E:degeneration-set}. So both the BHC and Painlev\'e IV degenerate at  \textit{exactly the same arithmetic relations on} $\theta_0,\, \theta_\infty$. That is, there is a \textit{complete correspondence} between the degeneration of monodromy/Stokes multipliers of the {BHE} as a connection, i.e., a biconfluent Heun connection (BHC), and the parameter space when Painlev\'e IV admits special solutions as characterised by Okamoto \cite{okamoto}, Noumi and Yamada \cite{noumi-yamada-1999, noumi2004}.   Moreover, we point out that these special solutions of Painlev\'e IV lie in the Picard-Viessot extension of $\partial^2+(x^2+\alpha)$ for some non-zero $\alpha$ (Theorem \ref{T:BHC-2}, and Theorem \ref {T:BHC-degenerate-3}),
	\item[(ii)] We sometimes adopt another set of parameters and write the \textit{general form of} BHE as
	\begin{equation}\label{E:BHE}
 	z y''(z)+(-2z^2+bz+c)y'(z)+(d+ez)y(z)=0.
	\end{equation}
where $b,c,d,e$ are parameters so that $(\al,\be,\ga,\de)=(c-1,-b,e+c+1,bc-2d)$. We show that the eigen-solutions to the {BHE} can assume the form
 		\begin{equation}\label{E:hautot_sum_0}
		y(x)=e^{x^2/4}\sum_{k=0}^N A_k D_{\frac{e}{2}-k}(x)=e^{x^2/4}\sum_{k=0}^N A_k D_{(\theta_\infty-\theta_0-1)-k}(x),\quad x=(b-2z)/\sqrt{2},
	\end{equation}
where  $e=(\gamma-\alpha-2)/2$ and the $D_\nu(x)$ is the \textit{parabolic cylinder function} (see Appendix B), first given by  Hautot \cite{hautot_1969}, \cite{hautot_1971} lie in certain invariant subspace $\mathcal{I}_N$ (see \S\ref{S:construction-subspace}) of dimension $N+1$ with respect to the BHE characterised by Picard-Viessot extension of $\partial^2+(x^2+\alpha)$, after choosing appropriate accessory parameters (Theorem \ref{E:eigen-equation}), at exactly the same monodromy/Stokes multipliers mentioned in (i) \textit{except at three cases}, and hence we provide an ``almost complete" correspondence between invariant subspaces of the {BHE} and the well-known special solutions of Painlev\'e IV equation again as characterised by Okamoto \cite{okamoto}, Noumi and Yamada \cite{noumi-yamada-1999, noumi2004}.

When the parameter $\alpha$ in BHE \eqref{E:BHEC}  $\alpha+1$ becomes an non-positive integer $-N\le0 $, then
we derive a second solution to \eqref{E:BHEC}
	\begin{equation}\label{E:hautot_sum_-1}
		g(x)=e^{x^2/4}\sum_{k=0}^N A_k E_{\frac{e}{2}-k}(x)=e^{x^2/4}\sum_{k=0}^N A_k E_{(\theta_\infty-\theta_0-1)-k}(x),\quad x=(b-2z)/\sqrt{2},
	\end{equation}
where $E_\nu(x)$ (see appendix B) can be regarded as the \textit{parabolic cylinder functions of the second kind}. The function $g(x)$ provides a second solution to the \eqref{E:BHEC}  linearly independent from \eqref{E:hautot_sum_0} under the assumption that $\alpha+1=-N$.

We note the proof of Theorem \ref{T:BHC-degenerate-3} can be completed after have written the  Hautot sums  \eqref{E:hautot_sum_0}  and our \eqref{E:hautot_sum_-1} are gauge equivalent to
	\[
		f_j(x)=p_{0,\, j}(x)\, e^{x^2/4}\,  D_{\frac{e}{2}-N+j}(x)+ p_{1,\, j}(x)\, e^{x^2/4}\, \big(D_{\frac{e}{2}-N+j}(x)\big)^\prime,\quad 0\le j\le N
	\]
and
	\[
		g_j(x)=p_{0,\, j}(x)\, e^{x^2/4}\,  E_{\frac{e}{2}-N+j}(x)+ p_{1,\, j}(x)\, e^{x^2/4}\, \big(E_{\frac{e}{2}-N+j}(x)\big)^\prime,\quad 0\le j\le N
	\]
for the same polynomials $p_{0,\, j}(x)$, $p_{1,\, j}(x)$ in both $f$ and $g$ in Theorem \ref{T:gauge} respectively. This implies that the regular singularity of \eqref{E:BHC} (resp. \eqref{E:BHEC}) at the origin becomes an apparent singularity. Hence the \eqref{E:BHC} (resp. \eqref{E:BHEC}) is gauge equivalent to a parabolic connection (Theorem \ref{T:BHC-degenerate-3}) (resp. parabolic equation).

Indeed  special function expansions similar to \eqref{E:hautot_sum_0}  for Fuchsian type (scalar) differential equations appeared in earlier works of Heine \cite{heine_1859} for the Lam\'e equation, and Kimura \cite{kimura_1970}, Erdelyi \cite{Erdelyi1}, Wolfrat \textit{et al} \cite{STW_2004} for the Heun equations. We refer the reader to \cite{CCT} for a correspondence between special solutions of the Darboux equation (which is an elliptic version of the Heun equation) and special solutions of Painlev\'e VI.


	\item[(iii)] to derive new general solutions of {BHE} each written, with rigorous justification, as an \textit{infinite sum} of parabolic cylinder functions
	\begin{equation}\label{E:infinite_sum}
		y(x)=e^{x^2/4}\sum_{k=0}^\infty A_k D_{\frac{e}{2}-k}(x)
		=e^{x^2/4}\sum_{k=0}^\infty A_k D_{(\theta_\infty-\theta_0-1)-k}(x)
,\quad x=(b-2z)/\sqrt{2}
	\end{equation}
that \textit{converges uniformly} in an half-plane (Theorem \ref{T:half-expansion}) and that each infinite sum of parabolic cylinder functions terminates into the eigen-solutions studied in part (ii). We show one can also construct an entire solution to \eqref{E:BHE} that converges in $\mathbb{C}$ by applying the symmetry group of  \eqref{E:BHE} (Theorem \ref{T:full-expansion}) from
Proposition \ref{T:BHE-symmetry}.
\end{enumerate}

We now further review fundamentals about the isomonodromy deformation of Painlev\'e IV as described in Jimbo and Miwa \cite{JM_1981b} which is our main reference in this paper. The Biconfluent Heun-type connection is a connection  over the rank two trivial vector bundle over the Riemann sphere $\mathbb{CP}^1$ with punctures at $x=0,\, \infty$. In addition to the normalised connection \eqref{E:BHC}, it follows from \cite[Appendix C]{JM_1981b} (see also \cite[p. 151]{FIKN}) that the  \textrm{BHC} admits asymptotic expansion of the form \cite{wasow}:
 	\begin{equation}\label{E:asy-0}
 		Y(x)\sim \Big(1+\frac{Y_1}{x}+\cdots\Big)\, e^{T(x)} 	
	\end{equation}
where
	\begin{equation}\label{E:asy-1}
		\begin{split}
		e^{T(x)} &=\begin{pmatrix}
			1 & 0 \\
			0 & -1
			\end{pmatrix}\, \frac{x^2}{2}+
			\begin{pmatrix}
			t & 0\\
			0 & -t
			\end{pmatrix}
			x+
			\begin{pmatrix}
				\theta_\infty & 0\\
				0 & -\theta_\infty
			\end{pmatrix}
			\, \log\frac{1}{x}\\
			&= \Big(1+O\big(\frac{1}{x}\big)\Big)
			\begin{pmatrix}
				e^{x^2/2+xt}\, x^{-\theta_\infty} & 0\\
				0 & e^{-x^2/2-xt}\, x^{\theta_\infty}
			\end{pmatrix}
		\end{split}
	\end{equation}
where
	\[
		Y_1(x)=\frac12
		\begin{pmatrix}
			-H_{IV} & -u\\
			2(z-\theta_0-\theta_\infty)/u & H_{IV}		
			\end{pmatrix},
	\]
and
	\[
		H_{IV}=H_{IV}(y,\, z;\, t)=
		\frac2y\, z^2-\big(y+2t+\frac{4\theta_0}{y}\big)z
		+(\theta_0+\theta_\infty)(y+2t).
	\]	

Moreover, the asymptotic behaviour of the expansion \eqref{E:asy-0}  together with \eqref{E:asy-1} in the sectors
	\[
		\Omega_1:\ \big(-\frac{3\pi}{4},\, \frac{\pi}{4}\big),
		\qquad
		\Omega_2:\ \big(-\frac{\pi}{4},\, \frac{3\pi}{4}\big),
		\qquad
		\Omega_3:\ \big(\frac{\pi}{4},\, \frac{5\pi}{4}\big)
		\qquad
		\Omega_4:\ \big(\frac{3\pi}{4},\, \frac{7\pi}{4}\big)
		\]
 labelled by $\Psi_k^{(\infty)}(x;\, t)\ (k=1,\, 2,\,3,\, 4)$  respectively, are related by Stokes matrices $S_k$ \cite{JM_1981b},  \cite[p. 2038]{mugan-fokas} (see also \cite[pp. 181--182]{FIKN})
	\[
		\Psi_{k+1}^{(\infty)}(x;\, t)=\Psi_k^{(\infty)}(x;\, t)\, S_k,\quad k=1,\, 2,\, 3;
	\]
	\[
		\Psi_1^{(\infty)}(x;\, t)=\Psi_4^{(\infty)}(x;\, t) \, S_4\, e^{2\pi i\theta_\infty\sigma_3},
		\quad
		\sigma_3=
		\begin{pmatrix}
			1 & 0\\
			0 & -1
		\end{pmatrix},
	\]
where
	\begin{equation}\label{E:stoke_matrices}
		S_1=
		\begin{pmatrix}
			1    & 0\\
			s_1 & 1
		\end{pmatrix},
		\quad
		S_2=
		\begin{pmatrix}
			1    & s_2\\
			0 & 1
		\end{pmatrix},
		\quad
		S_3=
		\begin{pmatrix}
			1   & 0\\
			s_3 & 1
		\end{pmatrix},
		\quad
		S_4=
		\begin{pmatrix}
			1   & s_4\\
			0 	& 1
		\end{pmatrix},
	\end{equation}
and the entries $s_k,\ k=1,\, 2,\, 3,\, 4$, called the (elements of) Stokes multipliers (matrices) , are related by
	\begin{equation}\label{E:stoke-multipliers}
		(1+s_2s_3)e^{2\pi i \theta_\infty}+[s_1s_4+(1+s_3s_4)(1+s_1s_2)]e^{-2\pi i\theta_\infty}
		=2\cos 2\pi \theta_0.
	\end{equation}

We now adopt
\begin{definition} \label{D:BHC-criteria} A Biconfluent Heun connection to be a connection of the form \eqref{E:BHC} such that the residue matrices in \eqref{E:BHC-matrices} meet the following criteria:
	\begin{enumerate}
		\item  the $A$ can be replaced by $-A$;
		\item the traceless matrix $C$ has eigenvalues $\pm\theta_0$;
		\item  and finally the matrix $B$ such that the diagonal matrix associated to the term $\log \frac{1}{x}$ in \eqref{E:asy-1} has eigenvalues $\pm\theta_\infty$.
	\end{enumerate}
\end{definition}

\begin{theorem}\label{T:BHC-2}
	Let $\Psi$ be a matrix valued function satisfying the $\mathrm{BHC}$ \eqref{E:BHC} and \eqref{E:BHC-matrices}.  Then $\theta_0\pm \theta_\infty \in \mathbb{Z}$ if and only if either the pair of Stokes matrices $S_1,\, S_3$ or the pair of Stokes matrices  $S_2,\, S_4$ in \eqref{E:stoke_matrices} reduces to identity matrices.
\end{theorem}

It follows from \eqref{E:matching-coeff} that the first condition in \eqref{E:degeneration-set} is equivalent to the commonly seen criterion
		\begin{equation}\label{E:P4_special_1}
			\beta=-2(2n+1+\varepsilon \alpha)^2,\quad \varepsilon=\pm 1.
		\end{equation}	
\begin{remark} We would like to mention that the above monodromy degeneration criterion alone does not guarantee, one needs to determine the appropriate eigenvalues before being able to write down the corresponding eigen-solutions.
\end{remark}
The correspondence between the second condition in \eqref{E:degeneration-set} and again the other commonly seen criterion
		\[
			  \beta=-2n^2, \quad n\in\mathbb{Z}
		\]
will be considered in the first part of the next theorem.

\begin{theorem}\label{T:BHC-degenerate-3}
The Biconfluent Heun connection \eqref{E:BHC}
	\begin{enumerate}
		\item[(i)] can be transformed from a parabolic type connection
			\begin{equation}\label{E:parabolic-connection}
				\frac{d\Phi}{dx}=\big(A^\prime\, x+B^\prime\big)\Phi,
			\end{equation}
	 by a Schlesinger (gauge) transformation only if $2\theta_0=n\in \mathbb{Z}$ (or equivalently $\beta=-2n^2$);
					\item[(ii)] the differential Galois group of the \eqref{E:BHC} is solvable only if $\theta_0\pm \theta_\infty \in \mathbb{Z}$ holds (or equivalently $\beta=-2(2n+1+\varepsilon \alpha)^2,\quad \varepsilon=\pm 1$\textrm{)}.;
		\item[(iii)] shares the same parameter space $(\theta_0,\, \theta_\infty)$ of Painlev\'e $\mathrm{IV}$  \eqref{E:P4} in that both equations admit solutions lying in the Picard-Viessot extension of $L=\partial^2+(x^2+\alpha)$ from the reductions of (i) and (ii).
	\end{enumerate}
\end{theorem}

We remark that the conclusion (i) corresponds to having the original biconfluent Heun connection can be transformed from a parabolic type connection via an appropriate Schlesinger (gauge) transformation. The conclusion (ii) essentially means that the entries of $\Psi$ can be expressed as a finite combinations of Hermite polynomials. This becomes explicit when we consider the corresponding biconfluent Heun equation.

This paper is organised as follows. We study the symmetries of the BHE and BHC in \S2 which will be used in the proof of Theorem \ref{T:BHC-degenerate-3}. In \S3 we prove the main results concerning the BHC. In particular we demonstrate that
the monodromy/Stokes multipliers of the BHC \textit{degenerates} as in Theorem \ref{T:BHC-degenerate-3} only if the parameter space $(\theta_0,\, \theta_\infty)$ corresponds to that of Painlev\'e IV equation $\mathrm{P_{IV}}$ via \eqref{E:matching-coeff}. The discussion of an algebraic structure of the BHC/BHE beyond the symmetry groups mentioned above which appears to be different from the affine Weyl group $\tilde{A}_2$ symmetry that is well-known for $\mathrm{P_{IV}}$ is beyond the scope of this paper. We continue our study of BHE in \S4 where we first show that the BHE admits an infinite expansion in terms of Parabolic functions in an half-plane. We prove its convergence by applying a recent asymptotic result on second order difference equations of Wong-Li \cite{wong}. We then demonstrate that the parabolic expansion terminates precisely when the coefficients in BHE correspond to
	\[
		2\,\theta_0\in\mathbb{Z},\qquad\textrm{or}
		\qquad
		\theta_0\pm \theta_\infty\in\mathbb{Z},
	\]
via \eqref{E:matching-coeff} with the exception of ``three straight lines'' passing through the origin of the $(\alpha,\beta)-$plane on which the ${\tilde{A}}^{(1)}_2$ lives.

\section{Symmetries}
\setcounter{equation}{0}

 Comparing (\ref{E:BHE}) with (\ref{E:BHEC}), one has
{so that}
 $$
 (b,c,d,e)=\big(-\be,\al+1,-\frac{1}{2}(\de+(1+\al)\be),\ga-\al-2\big).
 $$

 \begin{proposition}[\cite{maroni_a, maroni_b}]
 \label{T:BHE-symmetry}
 If we denote by $\phi_1(x)=y(\al,\be,\ga,\de;x)$ a solution of
 {\sl BHE} (\ref{E:BHEC}), then the following functions are also
 solutions of {\sl BHE}:
 \begin{eqnarray*}
 && \phi_2=z^{-\al}y(-\al,\be,\ga,\de;z) \\
 && \phi_3=y(\al,-\be,\ga, -\de; -z)\\
 && \phi_4=e^{\be z+ z^2}y(\al,-i \be,,-\ga,i \de; -i z) \\
 && \phi_5= e^{\be z+ z^2}y(\al,i \be,,-\ga,-i \de; i z) \\
 && \phi_6=z^{-\al}e^{\be z +z^2}y(-\al,-i \be,-\ga,i \de; -i z)\\
 && \phi_7=z^{-\al}e^{\be z+z^2}y(-\al,i \be,-\ga,-i \de; i z)\\
 && \phi_8=z^{-\al}y(-\al,-\be,\ga,-\de; -z).\\
 \end{eqnarray*}
 In particular, the symmetry group of the BHE is given by $C_2\times C_4$.
 \end{proposition}

We show the Biconfluent Heun connection also shares the symmetry group $C_2\times C_4$:
\begin{theorem}\label{T:BHC_sym}
	The $\mathrm{BHC}$ \eqref{E:BHC} has its symmetry group isomorphic to $C_2\times C_4$.
\end{theorem}

\begin{proof} Let $\mathcal{B}_{0,\, \infty}$ be the set of biconfluent Heun connections as defined in the Definition \ref{D:BHC-criteria}.  We define  $a:\, \mathcal{B}_{0,\, \infty}\longrightarrow \mathcal{B}_{0,\, \infty}$ be such that
	\[
		d+(Ax+B+C/x)\, dx \ \longmapsto\  d+(\tilde{A}x+\tilde{B}+\tilde{C}/x)\, dx,
	\]
where
	\[
	 	\theta_\infty\longmapsto -\theta_\infty,
		\qquad
		x\longmapsto ix
	\]
where $\pm\theta_\infty$ are eigenvalues of the diagonal matrix associated to the term $\log \frac{1}{x}$ in the corresponding \eqref{E:asy-1} for $\tilde{B}$ , i.e., that diagonal matrix becomes
		\[
			\begin{pmatrix}
				-\theta_\infty & 0\\
				0 & \theta_\infty
			\end{pmatrix}.
		\]
More precisely, we have
		\[
			\tilde{A}=
			\begin{pmatrix}
				-1 & 0\\
				0 & 1
			\end{pmatrix},
			\quad
			\tilde{B}=
			\begin{pmatrix}
			it & iu\\
			2(-z+\theta_0-\theta_\infty)/(iu), & -it
			\end{pmatrix}.
		\]
Indeed it can be verified that the ${B}$ becomes
	\[
		\begin{pmatrix}
			-t & -u\\
			2(-z+\theta_0+\theta_\infty)/u, & t
			\end{pmatrix}
		\quad\textrm{and}\quad
		\begin{pmatrix}
			-it & -iu\\
			2(-z+\theta_0-\theta_\infty)/(-iu), & it
			\end{pmatrix}
		\]
upon the actions of $a^2$ and $a^3$ respectively. Finally, taking into account of action of $a$ on the matrices $A$ and $C$, it is straightforward to check that $a^4=I$. We now define $b:\, \mathcal{B}_{0,\, \infty}\longrightarrow \mathcal{B}_{0,\, \infty}$ to be such that
	\begin{equation}\label{E:theta_0}
		\theta_0\longmapsto -\theta_0.
	\end{equation}
Clearly $b^2=I$. It is easy to verify that $ab=ba$. Hence the symmetry group of the \eqref{E:BHC} is isomorphic to $C_2\times C_4$ as desired.
\end{proof}


\section{Proof of Theorem \ref{T:BHC-2}}
\setcounter{equation}{0}
\begin{proof}  Suppose $S_1,\, S_3$ in \eqref{E:stoke-multipliers} reduce to identity matrices, i.e., $s_1=s_3=0$. Then the equation \eqref{E:stoke-multipliers} reduces to
	\[
		\cos 2\pi\, \theta_\infty=\cos 2\pi\, \theta_0.
		\]
A simple trigonometric argument shows that we must have $\theta_0\pm \theta_\infty=n$ for some integer $n$. If $S_2,\, S_4$ reduce to identity matrices, then one can also deduce the same conclusion with a similar argument. Conversely,
suppose $\theta_0\pm \theta_\infty=n$ is an integer. Then the equation \eqref{E:stoke-multipliers} becomes
	\begin{equation} \label{E:stoke-multipliers-2}
		(1+s_2s_3)e^{2\pi i \theta_\infty}+[s_1s_4+(1+s_3s_4)(1+s_1s_2)]e^{-2\pi i\theta_\infty}
		=2\cos 2\pi\, \theta_\infty.
	\end{equation}
To avoid a contradiction of the compatibility of real and imaginary parts on both sides, let us first assume that $s_2=0$ while $s_3\not=0$. Then
the equation \eqref{E:stoke-multipliers-2} becomes
	\[
			e^{2\pi i \theta_\infty}+[s_1s_4+(1+s_3s_4)]e^{-2\pi i\theta_\infty}
		=2\cos 2\pi\, \theta_\infty.
	\]
That is,
	\[
		s_4(s_1+s_3)=0.
	\]
Suppose $s_4\not=0$. Then $s_3=-s_1$. But then we have the matrix relation $S_3=S_1^{-1}=-S_1$. So the two matrices are identical under a projective change of coordinates. Hence $s_4=0$. That is, both $S_2,\, S_4$ reduce to identity matrices. If we now assume instead that $s_3=0$ and $s_2\not=0$, then by  a similar argument, we deduce $s_1=0$. Hence both $S_1,\, S_3$ are identity matrices.\\
\end{proof}

\section{Proof of Theorem \ref{T:BHC-degenerate-3}}
\setcounter{equation}{0}

We first prove parts (ii) and (iii) here. The proof of part (i) will be completed after the discussion of invariant subspace of BHE \eqref{E:BHE} in \S\ref{S:completion}.\\
We suppose that the differential Galois group of the \eqref{E:BHC} is solvable. Then it follows from the Kolchin's classification of differential Galois groups that the matrix representations of the algebraic subgroups $S_1,\cdots, S_4$ must belong to the category of triangular matrices \cite{kolchin}. Hence either $S_1,\, S_3$ or $S_2,\, S_4$ must reduce to identity matrices. It follows from part (i) above that $\theta_0\pm \theta_\infty=n$ is an integer. This completes the proof of part (ii).

\section{Invariant subspaces} 
\setcounter{equation}{0}

We now turn our attention to the \eqref{E:BHE} or equivalently \eqref{E:BHEC}.
Duval and Loday \cite[Prop. 13]{Duval_Loday1992}  applied the celebrated Kovacic algorithm \cite{Kovacic1986} to show that the BHE \eqref{E:BHEC} admits \textit{Liouvillian solutions} only when $\gamma-\alpha-2=2N$ for some integer $N\ge 1$. This implies that the BHE \eqref{E:BHEC} admits polynomial solutions which were previously obtained independently by Hautot \cite{hautot_1969}. Moreover, Hautot shows in \cite{hautot_1969} and in \cite{hautot_1971} that one could consider solutions of the BHE \eqref{E:BHE} in the form
	\begin{equation}\label{E:hautot_sum}
		y(x)=e^{x^2/4}\sum_{k=0}^N A_k D_{\frac{e}{2}-k}(x)=e^{x^2/4}\sum_{k=0}^N A_k D_{(\theta_\infty-\theta_0-1)-k}(x),\quad x=(b-2z)/\sqrt{2}
	\end{equation}
when $e=2N$ and $c=N$ respectively. See also \cite{ronv}. The condition $e=2N$ corresponds precisely the condition that $\gamma-\alpha-2=2N$ obtained by  Duval and Loday \cite[Prop. 13]{Duval_Loday1992}. Indeed, when $e=2N$, the sum \eqref{E:hautot_sum} reduces to
	\begin{equation}\label{E:hautot_sum_2}
		y(x)=\sum_{k=0}^N A_k H_{N-k}(x),\qquad x=(b-2z)/\sqrt{2},
	\end{equation}
where $H_k(x)$ denotes the Hermite polynomial of degree $k$, where, as we shall show below, that the coefficients $A_k$ satisfy a three-term recursion. Motivated by Hautot's work, we extend Hautot's finite sum \eqref{E:hautot_sum} into an infinite sum \eqref{E:infinite_sum}.
We shall show in \S\ref{S:Sums} with vigorous justification that this infinite sum does converge in any compact set in a half-plane under some mild condition on $b=-\beta$. Moreover that the infinite sum solution terminates exactly to \eqref{E:hautot_sum} or \eqref{E:hautot_sum_2} according to
	\begin{equation}
		\label{E:hautot_criteria}
			c=\alpha+1=2\theta_0=-N,\quad\textrm{and}\quad e=\gamma-\alpha-2=2(\theta_\infty-\theta_0-1)=2N
	\end{equation}
respectively, where  $N$ is a non-negative integer.  These two ``termination'' conditions \textit{essentially but not exactly} match those described in Theorem \ref{T:BHC-degenerate-3} (i) and (ii) respectively. Both types of finite-sums are due to Hautot \cite{hautot_1969, hautot_1971}.  See also \cite{ronv}. When $c=\alpha+1=2\theta_0=-N$ we find a linearly independent solution by replacing the parabolic  cylinder functions in
\eqref{E:hautot_sum} by parabolic cylinder functions of the second kind (Theorem \ref{T:gauge}).

We would like to recast Hautot's results in our invariant subspace framework in which the BHE \eqref{E:BHE} admits eigen-solutions when the parameter $\delta$, which plays the role of accessory parameter, as shown below, is appropriately chosen.   We first prove the following key lemma for this construction.

 \subsection{Construction of invariant subspaces}\label{S:construction-subspace}

 We shall show that one can consider those finite sums as certain eigen-solutions of the BHE when interpreted as a subspace in the vector space to be defined below. Moreover, we shall see that these finite form solutions, irrespective of which type, lie in the Picard-Viessot extension of the operator $\partial^2+(x^2+\alpha)$. We define
 	\[
 		\mathcal{I}_N:=\Big\{
 		e^{x^2/4}\sum_{k=0}^N A_k D_{e/2-k}(x):\ A_k\in\mathbb{C},\ x=-\sqrt{2}
 		(z-b/2)\Big\}
 	\]
and
\[
 		\mathcal{J}_N:=\Big\{
 		e^{x^2/4}\sum_{k=0}^N A_k E_{e/2-k}(x):\ A_k\in\mathbb{C},\ x=-\sqrt{2}
 		(z-b/2)\Big\}
 	\]where $e/2=N$ is an integer. Obviously, the $\mathcal{I}_N$ and $\mathcal{J}_N$ are vector-subspaces of the field of the Picard-Viessot extension of the operator $\partial^2+(x^2+\alpha)$.

\begin{theorem}\label{T:invariant_subspace} Let
	\begin{equation}\label{E:eigen-operator}
		-Lf(x)=(b-\sqrt{2}x)\,f^{\prime\prime}+\big(\sqrt{2}x^2-bx-\sqrt{2}c\big)\, f^\prime(x)+\big(-\frac{e}{\sqrt{2}}\,x+\frac{b\,e}{2}\big)\, f(x).
	\end{equation}
If either $e=2N$ or $c=-N$ for some integer $N$, then $L\mathcal{I}_N\subset \mathcal{I}_N$ and $L\mathcal{J}_N\subset \mathcal{J}_N$.
\end{theorem}

\begin{proof} It is sufficient to prove that
	\[
		L\big[e^{x^2/4} D_{\frac{e}{2}-k}(x)\big] \subset \textrm{span\ }
		\big\{e^{x^2/4} D_{\frac{e}{2}-k}(x)\big\}_{k=0}^N
		\]
for $0\le k\le N$.
We apply the formula \eqref{A:bateman_1}, \eqref{E:derived_parabolic_1} and \eqref{E:derived_parabolic_2} to obtain
	\begin{equation}\label{E:invariant_1}
		\begin{split}	
		(b-\sqrt{2}x) \frac{d^2}{dx^2}\Big[ e^{x^2/4} D_{\frac{e}{2}-k}(x)\Big]
		&=
		\big(\frac{e}{2}-k\big)\big(\frac{e}{2}-k-1\big)\Big[b e^{x^2/4} D_{\frac{e}{2}-k-2}(x)
		-\sqrt{2}x e^{x^2/4} D_{\frac{e}{2}-k-2}(x)\Big]\\
		&= \big(\frac{e}{2}-k\big) \big(\frac{e}{2}-k-1\big) e^{x^2/4}
		\Big[bD_{\frac{e}{2}-k-2}-\sqrt{2}
		\big(D_{\frac{e}{2}-k-1}(x)\\
		&\quad + \big(\frac{e}{2}-k-2\big) D_{\frac{e}{2}-k-3}\big)
		\Big]\\
		&=-\sqrt{2}\big(\frac{e}{2}-k\big) \big(\frac{e}{2}-k-1\big)  \big(\frac{e}{2}-k-2\big)e^{x^2/4} D_{\frac{e}{2}-k-3}\\
		&\quad +b\big(\frac{e}{2}-k\big) \big(\frac{e}{2}-k-1\big) e^{x^2/4} D_{\frac{e}{2}-k-2}(x)\\
		&\quad -\sqrt{2}\big(\frac{e}{2}-k\big)\big(\frac{e}{2}-k-1\big)e^{x^2/4} D_{\frac{e}{2}-k-1}(x).
	\end{split}
		\end{equation}
Similarly, we have
	\begin{equation}\label{E:invariant_2}
		\begin{split}
			\big(\sqrt{2}x^2-bx-\sqrt{2}c\big) &	\frac{d}{dx}\big[ e^{x^2/4} D_{\frac{e}{2}-k}(x)]
			=
			\big(\sqrt{2}x^2-bx-\sqrt{2}c\big) \big(\frac{e}{2}-k\big) e^{x^2/4}
D_{\frac{e}{2}-k-1}(x)\\
			&=\big(\frac{e}{2}-k\big)\, e^{x^2/4}
			\Big[\sqrt{2}
			x\Big(D_{\frac{e}{2}-k}+\big(\frac{e}{2}-k-1\big)D_{\frac{e}{2}-k-2}\Big)\\
			&\qquad -b \Big(D_{\frac{e}{2}-k}+\big(\frac{e}{2}-k-1\big)D_{\frac{e}{2}-k-2}\Big)-\sqrt{2}cD_{\frac{e}{2}-k-1}\Big]\\
			&= \big(\frac{e}{2}-k\big) e^{x^2/4}
			\Big[\sqrt{2}\Big(D_{\frac{e}{2}-k+1}+\big(\frac{e}{2}-k\Big)\, D_{\frac{e}{2}-k-1}\Big)\\
			&\qquad +\sqrt{2}\big(\frac{e}{2}-k-1\big)\Big(D_{\frac{e}{2}-k-1}+\big(\frac{e}{2}-k-2\big)D_{\frac{e}{2}-k-3}\Big)\\
			& \qquad -b \Big(D_{\frac{e}{2}-k}+\big(\frac{e}{2}-k-1\Big)D_{\frac{e}{2}-k-2}\big)-\sqrt{2}\,c\,D_{\frac{e}{2}-k-1}\Big]\\
			&=e^{x^2/4}\big(\frac{e}{2}-k\big)\Big\{\Big[ \sqrt{2}\big(\frac{e}{2}-k-1\big)+\sqrt{2}\big(\frac{e}{2}-k\big)-\sqrt{2} c\Big]
D_{\frac{e}{2}-k-1}
			\\
		&\qquad -b\big(\frac{e}{2}-k-1\big)D_{\frac{e}{2}-k-2} +\sqrt{2}D_{\frac{e}{2}-k+1}-bD_{\frac{e}{2}-k}\\
			&\qquad
			+\sqrt{2}\big(\frac{e}{2}-k-1\big))\big(\frac{e}{2}-k-2\big)D_{\frac{e}{2}-k-3}
			\Big\}\\
		\end{split}
	\end{equation}
and
	\begin{equation}\label{E:invariant_3}
		\begin{split}
			\Big(-\frac{e}{\sqrt{2}}\, x+\frac{be}{2}\Big)e^{x^2/4}D_{\frac{e}{2}-k}
			&=
			\big(-\frac{e}{\sqrt{2}}\big) \Big(D_{\frac{e}{2}-k+1}+\big(\frac{e}{2}-k\big)
			D_{\frac{e}{2}-k-1}\Big)e^{x^2/4}\\
			&\qquad + \frac{be}{2} e^{x^2/4}D_{\frac{e}{2}-k}\\
			&=\Big[\big(-\frac{e}{\sqrt{2}}\big) D_{\frac{e}{2}-k+1}-
			\big(-\frac{e}{\sqrt{2}}\big) \big(\frac{e}{2}-k\big) D_{\frac{e}{2}-k-1}+
			\frac{be}{2}
			D_{\frac{e}{2}-k}\Big]e^{x^2/4}
		\end{split}
	\end{equation}
			
We see from combining the expressions \eqref{E:invariant_1}, \eqref{E:invariant_2} and \eqref{E:invariant_3} that
	\begin{equation}\label{E:invariant_4}
		\begin{split}
			L \big[e^{x^2/4} D_{\frac{e}{2}-k}(x)\big]				
			&=
			-\sqrt{2}\big(\frac{e}{2}-k\big)(c+k)\, e^{x^2/4} D_{\frac{e}{2}-k-1}(x)+\\
			&\qquad +\big\{\textrm{linear combination of terms involving\ }
			e^{x^2/4}D_{\frac{e}{2}-j-1}(x)\textrm{\ with\ } j\le k\big\}\\
			&\in  \mathrm{Span}_\mathbb{C}  \big\{e^{x^2/4} D_{\frac{e}{2}-j}:\  0\le j\le N \big\}
		\end{split}
	\end{equation}
provided that the factor of $D_{\frac{e}{2}-k-1}(x)$ in \eqref{E:invariant_3} vanishes when $k=N$. Since this holds for each $k\, (0\le k\le N)$.
That is, either $e=2N$ or $c=-N$. So we have proved $L\mathcal{I}_N\subset \mathcal{I}_N$. Since the parabolic cylinder function of the second kind $E_\nu(x)$ defined and satisfies all the same recursion formulae listed in Appendix B, so the proof of $L\mathcal{J}_N\subset \mathcal{J}_N$ becomes verbatim to the above proof which we skip.
\end{proof}

The theorem shows that the operator $L$
is a linear transformation of a finite dimensional vector spaces $\mathcal{I}_N$ and $\mathcal{J}_N$ . So the Theorem \ref{T:invariant_subspace} discussed above can be considered as proper eigen-value problem $Lf=d\, f$ for $f\in\mathcal{I}_N$ and $\mathcal{J}_N$  must have a solution when either $e=2N$ or $c=N$ for some integer $N$. Indeed, it follows from the proof of the above Theorem \ref{T:invariant_subspace} that there exists $N+1$ eigen-solutions $f_k$ and $g_k$ each corresponds to an eigenvalue $d_k, \ k=0,\, 1,\, \cdots, N$ in $\mathcal{I}_N$ and $\mathcal{J}_N$ and respectively.

\subsection{Invariant subspaces and correspondence with $P_\mathrm{IV}$}

The last subsection puts Hautot's sums and our sum for parabolic cylinder functions of second kind as proper eigenvalue problems in  suitably defined vector spaces. The ranges of $\alpha$ and $\gamma$ (resp. $\theta_0$ and $\theta_\infty$) that are allowed in \eqref{E:hautot_criteria} are \textit{incomplete} when compared to those
described in Theorem \ref{T:BHC-degenerate-3}, namely, $\alpha$ and $\gamma$ (resp. $\theta_0$ and $\theta_\infty$) that are allowed to vary amongst all possible positive and negative values. The following Figure 1 depicts the ranges of $\gamma$ and $\alpha$ in a graphical manner.

\begin{figure}[ht]
			\centering
			\includegraphics[width=0.45\textwidth]{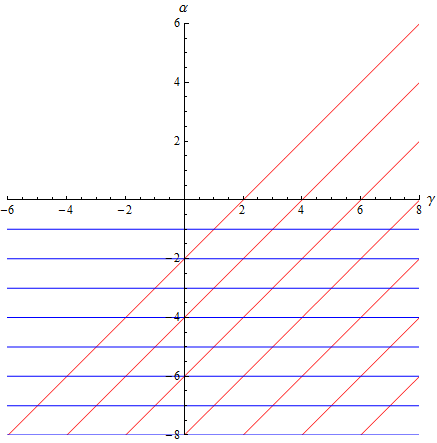}
			\caption{The above shows that $\gamma=-1,\, -2,\, -3,\, \cdots$, and $\gamma-\alpha-2=2,\, 4,\, 6,\, \cdots$ in the $(\gamma,\, \alpha)-$plane}
		\end{figure}
\bigskip

The collection of all $(\alpha,\, \gamma)$ from $\gamma=-1,\, -2,\, -3,\, \cdots$, and $\gamma-\alpha-2=2,\, 4,\, 6,\, \cdots$ in the $(\gamma,\, \alpha)-$plane are a consequence of the degeneration of monodromy/Stokes multipliers of the \eqref{E:BHC}. We shall show that apart from three-exceptions one can ``tile up" the whole  $(\gamma,\, \alpha)-$plane by applying the symmetry of the \eqref{E:BHEC} as the follow Figure 2.
\begin{figure}[ht]
			\centering
			\includegraphics[width=0.45\textwidth]{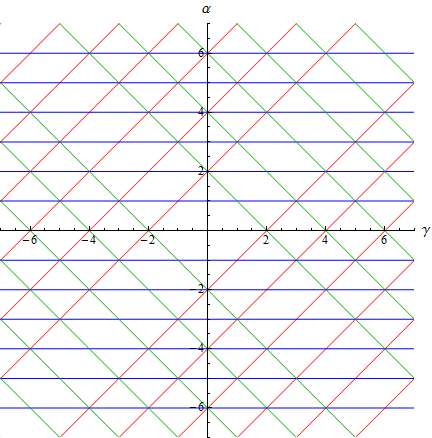}
			\caption{This figure shows that degeneration of monodromy/Stokes multiplier of \eqref{E:BHEC} in the $(\gamma,\, \alpha)-$plane}
		\end{figure}
\bigskip

One can eventually ``recover'' the ``missing three lines'' of the above figure by considering the degeneration of monodromy/Stokes multiplers of the BHC \eqref{E:BHC} so that there is a complete correspondence between the $(\alpha,\, \gamma)$ of \eqref{E:BHE} and $(\xi,\, \eta)$ of the Painlev\'e IV equation.
\begin{figure}[ht]
			\centering
			\includegraphics[width=0.45\textwidth]{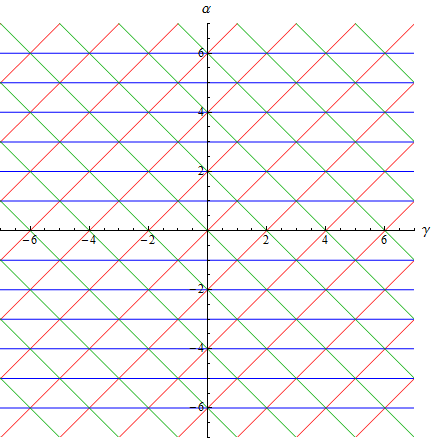}
			\caption{This figure shows that degeneration of monodromy/Stokes multiplier of \eqref{E:BHC} in the $(\gamma,\, \alpha)-$plane}
		\end{figure}

We first derive the three-term recursion of the coefficients $A_k$
that appears in the formal infinite sum \eqref{E:infinite_sum}.


\begin{theorem}\label{T:formal-series}
Let $f(x)$ be the formal series solution defined in \eqref{E:infinite_sum}  to
	\begin{equation}\label{E:operator-eqn}
		Lf=df,
	\end{equation}
where $L$ is the operator \eqref {E:eigen-operator}.  Then the coefficients $A_n$
 satisfies the recurrence relation
 \begin{equation} \label{E:3-recursion}
 \al_n A_{n+1}+\be_n A_n +\ga_n A_{n-1}=0,
 \end{equation}
 where
 \begin{equation}\label{E:3-term}
 \al_n=-\sqrt{2}(n+1),\quad \be_n=d+bn,\quad
 \ga_n=-\sqrt{2}(n+c-1)(\frac{e}{2}-n+1).
 \end{equation}
 Moreover, $A_0$ is arbitrary and $\d A_1=\frac{d}{\sqrt{2}}A_0$. Moreover, the formal expansion
	\begin{equation}\label{E:2nd-expansion}
		g(x)=e^{x^2/4}\sum_{k=0}^\infty A_kE_{\frac{e}{2}-k}(x)=
		e^{x^2/4}\sum_{k=0}^\infty A_kE_{(\theta_\infty-\theta_0-1-k)}(x),
	\end{equation}
where the $E_\nu(x)$ is defined in \eqref{E:parabolic_2nd_soln},
serves as a linear independent solution to the operator equation \eqref {E:operator-eqn} and the coefficient $A_k$ also satisfy the same recursions \eqref{E:3-recursion} and \eqref{E:3-term}.
\end{theorem}

 \begin{proof}
We first note that the parabolic cylinder function $D_\nu$ satisfies the differential equation
 \begin{equation}\label{E:parabolic}
  D_\nu''(x)=(\frac{1}{4}x^2-\nu-\frac{1}{2})D_\nu(x).
  \end{equation}

We substitute the formal sum \eqref{E:infinite_sum} into the equation \eqref{E:eigen-equation-2} yields
\begin{eqnarray}
 \lefteqn{
 (b-\sqrt{2}x)\sum_{0}^\infty A_n
 D_{\frac{e}{2}-n}''(x)-\sqrt{2}c\sum_{0}^\infty A_n
 D_{\frac{e}{2}-n}'(x)}\\
 &&+\left(
 (b-\sqrt{2}{x})(\frac{1}{2}-\frac{x^2}{4})-\frac{e+c}{\sqrt{2}}x+d+\frac{b
 e}{2}\right)\, \sum_{0}^\infty A_n D_{\frac{e}{2}-n}(x)=0.
 \label{eq2.2}
 \end{eqnarray}

 Then by \eqref{E:derived_parabolic_1} and \label{E:derived_parabolic_2}, we have, after simplification,
 $$
 \sum_{n=0}^\infty \left( -\sqrt{2}n A_n
 D_{\frac{e}{2}-n+1}(x)+(d+bn)A_n
 D_{\frac{e}{2}-n}(x)-\sqrt{2}(n+c)(\frac{e}{2}-n) A_n
 D_{\frac{e}{2}-n-1}(x)\right)=0
 $$
 Let $k=n-1\geq 1$. We compare the coefficients of $\d
 D_{\frac{e}{2}-k}(x)$ and thus obtain the recurrence relation
 	\begin{equation}\label{E:critical-term}
 -\sqrt{2}(k+1)A_{k+1}+(d+bk)A_k-\sqrt{2}(c+k-1)(\frac{e}{2}-k+1)A_{k-1}=0.
 	\end{equation}
 Thus \eqref{E:3-recursion} and \eqref{E:3-term} are valid.
 It is also clear that $A_0$ is arbitrary, and comparing coefficients of
 $D_\frac{e}{2}(x)$ gives $\d A_1=\frac{d}{\sqrt{2}}A_0$.
 \begin{flushright}
 \end{flushright}

 As for the second linearly independent expansion \eqref{E:2nd-expansion} it is sufficient to note that the $E_\nu(x)$ satisfies the same recursions
	\begin{eqnarray}
  2 E_\nu'(x)&=&-E_{\nu+1}(x)+\nu E_{\nu}(x),\label{E:derived_parabolic_3}\\
  x E_\nu(x) &=& E_{\nu+1}(x)+\nu E_{\nu-1}(x)\label{E:derived_parabolic_4}
  \end{eqnarray}
as  those for $D_\nu(x)$  in \eqref{E:derived_parabolic_1} and \eqref{E:derived_parabolic_2} respectively. The remaining steps in verifying the expansion \eqref{E:2nd-expansion} indeed satisfies the \eqref{E:operator-eqn} with the help of \eqref{E:derived_parabolic_3} and \eqref{E:derived_parabolic_4} are the same to those of the first expansion \eqref{E:infinite_sum} just verified above. So we skip the details.
\end{proof}

The cases when $N\ge 1$ in the following theorem were due to Hautot. Since they are not well-known, so we shall reproduce  the derivation of the two types of solutions from our new prospective and in a consolidated manner. In particular, we shall obtain those solutions for different combination of signs of $\alpha,\, \beta,\, \gamma,\, \delta$ by the symmetry of the BHE.

\begin{theorem} \label{T:invariant_map}
Let $\mathcal{L}: \mathcal{I}_N\longrightarrow \mathcal{I}_N$ be defined by
	\begin{equation}\label{E:eigen-equation}
		\mathcal{L}f(x)=xf^{\prime\prime}(x)+(1+\alpha-\beta x-2x^2)f^\prime(x)+[(\gamma-\alpha-2)x]f(x).	
	\end{equation}
If either
	\begin{equation}\label{E:new_criteria}
	\mathrm{(i)\quad }	\alpha\pm \gamma =2N\quad (N\in\mathbb{Z}\backslash\{0\})\quad\textrm{or}\quad
	\mathrm{(ii)\quad }	\alpha = N \quad (N\in\mathbb{Z}\backslash\{0\}),
	\end{equation}
then for each non-zero integer $N$ and each of the following four cases, there is an aggregate of $N+1$ eigenvalues $\delta$, and the
	\begin{equation}\label{E:eigen-equation-2}
		\mathcal{L} f= -df
, \qquad d=-\frac12\big(\delta+ (1+\alpha)\beta\big)
	\end{equation}
admits an aggregate of $N+1$ eigen-solutions, respectively, of the form
	\begin{enumerate}
		\item[(I)] using $\phi_1(x)$ in Theorem \ref{T:BHE-symmetry},
	\begin{equation}\label{E:hautot_sum_3}
		f_j(x)=e^{x^2/4}\sum_{k=0}^N A_{k,\, j}\, D_{\frac{e}{2}-k}(x),
		\quad 0\le j\le N
	\end{equation}
where $e=\gamma-\alpha-2=2N$ or $1+\alpha = -N\ (N\ge 1)$ and the $A_{k,\, j}=A_{k,\, j} (\alpha,\, \beta,\, \gamma,\, \delta_j)$ are given by \eqref {E:3-recursion},
or
	\item[(II)] using $\phi_2(x)$ in Theorem \ref{T:BHE-symmetry},
	\begin{equation}\label{E:hautot_sum_4}
		f_j(x)
		=x^{-\alpha}e^{x^2/4}\sum_{k=0}^N A_{k,\, j}\, D_{\frac{e}{2}-k}(x),
		\quad 0\le j\le N
	\end{equation}
where $e=\gamma+\alpha-2=2N$, $1-\alpha=-N \ (N\ge 1)$, the $A_{k,\, j}=A_{k,\, j} (-\alpha,\, \beta,\, \gamma,\, \delta_j)$ are given by \eqref {E:3-recursion}
or
	\item[(III)] using $\phi_5(x)$ in Theorem \ref{T:BHE-symmetry},
	\begin{equation}\label{E:hautot_sum_5}
		f_j(x)
		=e^{\beta x+3x^2/4}\sum_{k=0}^N A_{k,\, j}\, D_{\frac{e}{2}-k}(x),
	\end{equation}
$0\le j\le N$  where $e=-\gamma-\alpha-2=2N$, $1+\alpha=-N \ (N\ge 1)$, $A_{k\, j}=A_{k,\, j} (\alpha,\, i\beta,\, -\gamma,\, -i\delta_j)$
or
	\item[(IV)] using $\phi_7(x)$ in Theorem \ref{T:BHE-symmetry},
	\begin{equation}\label{E:hautot_sum_6}
		f_j(x)
		=x^{-\alpha}e^{\beta x+3x^2/4}\sum_{k=0}^N A_{k,\, j}\, D_{\frac{e}{2}-k}(x),
	\end{equation}
$0\le j\le N$  where $e=-\gamma+\alpha-2=2N$, $1-\alpha=-N \ (N\ge 1)$,  $A_{k,\, j}=A_{k,\, j} (-\alpha,\, i\beta,\, -\gamma,\, -i\delta_j)$
\end{enumerate}
where each set of coefficients $A_k$, of their respective choices of $\alpha,\, \beta,\, \gamma,\, \delta$, satisfies a three-term recursion \eqref{E:3-term}.

	 Moreover, if $\alpha+1>0$ in the four cases above, then the $n+1$ respectively eigenvalues are necessarily real and distinct.
 \end{theorem}

\begin{proof}  We substitute the formal sum \eqref{E:infinite_sum} into the equation \eqref{E:eigen-equation-2} and apply Theorem \ref{T:formal-series}. Let $\gamma-\alpha-2=2N$ or $c=\alpha+1=N$. Then the last term of \eqref{E:critical-term} vanishes and the three-term recursion \eqref {E:3-recursion} terminates provided that the determinant $\det \Phi=0$ where
$$
 \Phi=\left[
 \begin{array}{ccccccc}
  {d} & {-\sqrt{2}}  & 0 &\cdots &0\\
  -{c e}/{\sqrt{2} } & d+b & -2\sqrt{2} &\cdots &0\\
  \vdots&  \vdots&\vdots & \vdots & \vdots\\
  0 &  \cdots& {-\frac{1}{\sqrt{2}}(e-2N-4)(c+N-2)} & {d+b (N-1)} & -(N+1)\sqrt{2}\\
  0 &  \cdots &0 & -\frac{1}{\sqrt{2}}(e-2N+2)(c+N-1) & d+b N
  \end{array}
  \right]
  $$
to guarantee  a system of $N+1$ linear equations in $(A_0,\ldots,A_N)$ to have a non-trivial solution, while from \eqref{E:3-recursion}--\eqref{E:3-term}, $A_i=0$ for all $i\geq N+1$.
  Furthermore, the determinant vanishes if a ${(N+1)}$th-degree polynomial in $d$, and so has {$N+1$} roots in $d$.
  (That is why we say $d$ is an eigenvalue of (\ref{eq2.9}).) Also when $e=2N$, the solution $f$ reduces to a Hermite
  polynomial of degree $N$.

On the other hand, if $d=0=c$, then $0=A_1=A_2=\cdots$, when $f(x)= \rme^{x^2/4} D_{\frac{e}{2}}(x)$ is a solution.
  Therefore we conclude that {BHE} has a parabolic cylinder type solution if $c=-N\ (N\in \bfN\cup \{0\})$, {or
  equivalently $\al=c-1 $ is a negative integer.}  This completes the derivation of \eqref{E:hautot_sum_3}. Moreover,
 we deduce from a theorem of Rovder \cite{rovder} that if $\alpha+1>0$, then all the eigenvalues $\delta$ are real and distinct.

We now apply the symmetry of the \eqref{E:BHE} from \eqref{T:BHE-symmetry} similar from the above argument that $\phi_2(x)$, $\phi_5(x)$ and $\phi_7(x)$  assumes the forms \eqref{E:hautot_sum_4}, \eqref{E:hautot_sum_5} and \eqref{E:hautot_sum_6} for the respective restriction of $\alpha,\, \beta,\, \gamma$.
\end{proof}

\section{Invariant subspaces and apparent singularity}
\setcounter{equation}{0}

The following theorem shows that the special solutions \eqref{E:hautot_sum_3} obtained from \eqref{E:eigen-equation} can be written in terms of the pair $e^{x^2/4} D_{\frac{e}{2}-N}$ and $e^{x^2/4} \big(D_{\frac{e}{2}-N}\big)^\prime$ with polynomial coefficients via a change of basis, i.e., a gauge transformation. In particular, if $1+\alpha =-N\le 0$, then the \eqref{E:eigen-equation} can be transformed to a parabolic equation. That is, the regular singularity of \eqref{E:BHEC} at $x=0$ becomes an \textit{apparent singularity} of \eqref{E:eigen-equation} at $x=0$.

\begin{theorem}\label{T:gauge}
 \begin{enumerate}
	\item Let $e=\gamma-\alpha-2=2N$ or $c=\alpha+1=-N$ for some integer $N\ge 1$. Suppose
\begin{equation}\label{E:hautot_sum_7}
		f_j(x)=e^{x^2/4}\sum_{k=0}^N A_{k,\, j}\, D_{\frac{e}{2}-k}(x),
		\quad 0\le j\le N
	\end{equation}
are eigen-solutions to \eqref{E:eigen-equation}. Then
	\begin{equation}\label{E:gauge_down}
		f_j(x)=p_{0,\, j}(x)\, e^{x^2/4} D_{\frac{e}{2}-N}(x)+ p_{1,\, j}(x)\, e^{x^2/4}\big(D_{\frac{e}{2}-N}(x)\big)^\prime,
		\quad 0\le j\le N
	\end{equation}
for some polynomials $p_{0,\, j}(x)$ and $p_{1,\, j}(x)$. Moreover, if $\alpha+1=-N$ holds, then the \eqref{E:eigen-equation} admits a linearly independent solution $g_j(x)$ to the $f_j(x)$ so that each $D_\nu(x)$ in \eqref{E:hautot_sum_7} is replaced by $E_\nu(x)$. Each $g_j(x)\ (0\le j \le N)$ can also be written in the form
	\begin{equation}\label{E:gauge_down_b}
		g_j(x)=p_{0,\, j}(x)\, e^{x^2/4} E_{\frac{e}{2}-N}(x)+ p_{1,\, j}(x)\, e^{x^2/4}\big(E_{\frac{e}{2}-N}(x)\big)^\prime,
		\quad 0\le j\le N
	\end{equation}
for the same polynomials $p_{0,\, j}(x)$ and $p_{1,\, j}(x)$.  We deduce that either the gauge transformation \eqref{E:gauge_down} or \eqref{E:gauge_down_b}  transforms the parabolic cylinder equation
	\begin{equation}\label{E:hermite_1}
		y^{\prime\prime}+\Big[\big(\frac{e}{2}-N\big)+\frac12-\frac14x^2\Big]y=0
	\end{equation}
to the \eqref{E:eigen-equation-2}. Furthermore, there are polynomials $P_{0,\, j},\, P_{1,\, j}$ such that
	\begin{equation}\label{E:gauge_up}
		f_j(x)=P_{0,\, j}(x)\, e^{x^2/4} D_{\frac{e}{2}}(x)+ P_{1,\, j}(x)\, e^{x^2/4}\big(D_{\frac{e}{2}}(x)\big)^\prime,
\quad 0\le j\le N.
	\end{equation}
and
	\begin{equation}\label{E:gauge_up_b}
		g_j(x)=P_{0,\, j}(x)\, e^{x^2/4} E_{\frac{e}{2}}(x)+ P_{1,\, j}(x)\, e^{x^2/4}\big(E_{\frac{e}{2}}(x)\big)^\prime,
\quad 0\le j\le N.
	\end{equation}
That is, either the gauge transformation \eqref{E:gauge_up} or \eqref{E:gauge_up_b} transforms	
	\begin{equation}\label{E:hermite_2}
		y^{\prime\prime}+\Big[\frac{e+1}{2}-\frac14x^2\Big]y=0
	\end{equation}
to the \eqref{E:eigen-equation-2}.
	\item Let $e=\gamma+\alpha-2=2N$ or $c=-\alpha+1=-N$ for some integer $N\ge 1$. Then \eqref {E:hautot_sum_4} can be written as
	\begin{equation}\label{E:gauge_down_2}
		f_j(x)=p_{0,\, j}(x)\, x^{-\alpha}e^{x^2/4} D_{\frac{e}{2}-N}(x)+ p_{1,\, j}(x)\, x^{-\alpha} e^{x^2/4}\big(D_{\frac{e}{2}-N}(x)\big)^\prime,
		\quad 0\le j\le N
	\end{equation}
\eqref{E:gauge_down} for some polynomials $p_{0,\, j}(x)$ and $p_{1,\, j}(x)$. Moreover, if $-\alpha+1=-N$ holds,  then the \eqref{E:eigen-equation} admits a linearly independent solution $g_j(x)$ to each of the $f_j(x)$ with  $D_\nu(x)$ in \eqref{E:hautot_sum_7} replaced by $E_\nu(x)$. Each $g_j(x)\ (0\le j \le N)$ can also be written in the form
	\begin{equation}\label{E:gauge_down_2b}
		g_j(x)=p_{0,\, j}(x)\, e^{x^2/4} E_{\frac{e}{2}-N}(x)+ p_{1,\, j}(x)\, e^{x^2/4}\big(E_{\frac{e}{2}-N}(x)\big)^\prime,
		\quad 0\le j\le N
	\end{equation}
for the same polynomials $q_{0,\, j}(x)$ and $q_{1,\, j}(x)$.  We deduce that either the gauge transformation \eqref{E:gauge_down_2} or \eqref{E:gauge_down_2b}  transforms the parabolic cylinder equation
 \eqref{E:hermite_1}
to the \eqref{E:eigen-equation-2}. Furthermore, there are polynomials $P_{0,\, j},\, P_{1,\, j}$ such that
		\begin{equation}\label{E:gauge_up_2}
		f_j(x)=P_{0,\, j}(x)\, x^{-\alpha}e^{x^2/4} D_{\frac{e}{2}}(x)+ P_{1,\, j}(x)\, x^{-\alpha} e^{x^2/4}\big(D_{\frac{e}{2}}(x)\big)^\prime,
	\quad 0\le j\le N
	\end{equation}
and
	\begin{equation}\label{E:gauge_up_2b}
		g_j(x)=P_{0,\, j}(x)\, x^{-\alpha}e^{x^2/4} E_{\frac{e}{2}}(x)+ P_{1,\, j}(x)\, x^{-\alpha} e^{x^2/4}\big(E_{\frac{e}{2}}(x)\big)^\prime,
	\quad 0\le j\le N.
	\end{equation}
That is, either the gauge transformation \eqref{E:gauge_up_2} or \eqref{E:gauge_up_2b}
 transforms the parabolic cylinder equation \eqref{E:hermite_2} to the \eqref{E:eigen-equation-2}.
	\item Let $e=-\gamma-\alpha-2=2N$, and $\alpha+1=-N, \ (N\ge 1)$. Then the \eqref{E:hautot_sum_5} can be written as
			\begin{equation}\label{E:gauge_down_3}
		f_j(x)=p_{0,\, j}(x)\, e^{\beta x+3x^2/4} D_{\frac{e}{2}-N}(x)+ p_{1,\, j}(x)\, e^ {\beta x+3x^2/4}\big(D_{\frac{e}{2}-N}(x)\big)^\prime,
\quad 0\le j\le N.
	\end{equation}
Moreover, if $\alpha+1=-N\ (N\ge 1)$ holds,
 then the \eqref{E:eigen-equation} admits a linearly independent solution $g_j(x)$ to each of the $f_j(x)$ with  $D_\nu(x)$ in \eqref{E:hautot_sum_7} replaced by $E_\nu(x)$. Each $g_j(x)\ (0\le j \le N)$ can also be written in the form
	\begin{equation}\label{E:gauge_down_3b}
		g_j(x)=P_{0,\, j}(x)\, e^{x^2/4} E_{\frac{e}{2}-N}(x)+ P_{1,\, j}(x)\, e^{x^2/4}\big(E_{\frac{e}{2}-N}(x)\big)^\prime,
		\quad 0\le j\le N
	\end{equation}
for the same polynomials $p_{0,\, j}(x)$ and $p_{1,\, j}(x)$.  We deduce that either the gauge transformation \eqref {E:gauge_down_3}  or \eqref {E:gauge_down_3b} transforms the corresponding parabolic cylinder equation  \eqref {E:hermite_1} to \eqref{E:eigen-equation-2}. Furthermore, there are polynomials $P_{0,\, j},\, P_{1,\, j}$ such that
	 \begin{equation}\label{E:gauge_up_3}
		f_j(x)=P_{0,\, j}(x)\, e^{\beta x+3x^2/4} D_{\frac{e}{2}}(x)+ P_{1,\, j}(x)\, e^{\beta x+3x^2/4}\big(D_{\frac{e}{2}}(x)\big)^\prime,
\quad 0\le j\le N
	\end{equation}
and
	 \begin{equation}\label{E:gauge_up_3b}
		g_j(x)=P_{0,\, j}(x)\, e^{\beta x+3x^2/4} E_{\frac{e}{2}}(x)+ P_{1,\, j}(x)\, e^{\beta x+3x^2/4}\big(E_{\frac{e}{2}}(x)\big)^\prime,
\quad 0\le j\le N.
	\end{equation}
That is, the gauge transformation \eqref{E:gauge_up_3} or \eqref{E:gauge_up_3b} transforms the corresponding Hermite equation \eqref{E:hermite_2} to the \eqref{E:eigen-equation-2}.
 		\item Let $e=-\gamma+\alpha-2=2N$ and $-\alpha+1=-N\ (N\ge 1)$. Then the \eqref{E:hautot_sum_7} can be written as
			\begin{equation}\label{E:gauge_down_4}
		f_j(x)=p_{0,\, j}(x)\, x^{-\alpha}e^{\beta x+3x^2/4} D_{\frac{e}{2}-N}(x)+ p_{1,\, j}(x)\, x^{-\alpha} e^ {\beta x+3x^2/4}\big(D_{\frac{e}{2}-N}(x)\big)^\prime,
\quad 0\le j\le N.
			\end{equation}
		Moreover, if $-\alpha+1=-N\ (N\ge 1)$,  then the \eqref{E:eigen-equation} admits a linearly independent solution $g_j(x)$ to each of the $f_j(x)$ with  $D_\nu(x)$ in \eqref{E:hautot_sum_7} replaced by $E_\nu(x)$. Each $g_j(x)\ (0\le j \le N)$ can also be written in the form
	\begin{equation}\label{E:gauge_down_4b}
		g_j(x)=p_{0,\, j}(x)\, e^{x^2/4} E_{\frac{e}{2}-N}(x)+ p_{1,\, j}(x)\, e^{x^2/4}\big(E_{\frac{e}{2}-N}(x)\big)^\prime,
		\quad 0\le j\le N
	\end{equation}
for the same polynomials $p_{0,\, j}(x)$ and $p_{1,\, j}(x)$.  We deduce that either
the gauge transformation \eqref {E:gauge_down_4} or \eqref {E:gauge_down_4b} transforms the corresponding parabolic cylinder equation  \eqref {E:hermite_1} to \eqref{E:eigen-equation-2}. Furthermore, there are polynomials $P_{0,\, j},\, P_{,\, j}$ such that
	 \begin{equation}\label{E:gauge_up_4}
		f_j(x)=P_{0,\, j}(x)\,x^{-\alpha} e^{\beta x+3x^2/4} D_{\frac{e}{2}}(x)+ P_{1,\, j}(x)\, x^{-\alpha} e^{\beta x+3x^2/4}\big(D_{\frac{e}{2}}(x)\big)^\prime,
\quad 0\le j\le N
	\end{equation}
and
	 \begin{equation}\label{E:gauge_up_4b}
		g_j(x)=P_{0,\, j}(x)\,x^{-\alpha} e^{\beta x+3x^2/4} E_{\frac{e}{2}}(x)+ P_{1,\, j}(x)\, x^{-\alpha} e^{\beta x+3x^2/4}\big(E_{\frac{e}{2}}(x)\big)^\prime,
\quad 0\le j\le N.
	\end{equation}
That is, either the gauge transformation \eqref{E:gauge_up_4} and \eqref{E:gauge_up_4b} transforms the corresponding parabolic cylinder equation	\eqref{E:hermite_2} to the \eqref{E:eigen-equation-2}.
	\end{enumerate}
\end{theorem}

\begin{proof} Since both the parabolic cylinder function $D_\nu(x)$ and its second kind $E_\nu(x)$ satisfy exactly the same differential-difference equations \eqref{A:bateman_1}, \eqref{A:bateman_2}, \eqref{E:derived_parabolic_1} and \eqref{E:derived_parabolic_2}, so it suffices to prove the above statement for the $D_\nu(x)$ only.

We apply induction on $k$. Let $k=1$ and we apply the two identities involving parabolic cylinder functions from \eqref{E:derived_parabolic_1} and \eqref{E:derived_parabolic_2} to yield
	\begin{equation}\label{E:first-step}
		\begin{split}			
		f(x)=&
			A_0e^{x^2/4}D_{\frac{e}{2}}+A_1 e^{x^2/4}D_{\frac{e}{2}-1}= e^{x^2/4}\big[A_0 D_{\frac{e}{2}}+A_1 D_{\frac{e}{2}-1}\big]\\
			&= e^{x^2/4}\Big[A_0\Big(\big(\frac{e}{2}-1\big)D_{\frac{e}{2}-1}-2 D_{\frac{e}{2}-1}^\prime\Big)+A_1 D_{\frac{e}{2}-1}\Big]\\
			&= e^{x^2/4}\Big[\Big(A_0\big(\frac{e}{2}-1\big)+A_1\Big) D_{\frac{e}{2}-1}-2 D_{\frac{e}{2}-1}^\prime\Big]
							\end{split}
	\end{equation}
as desired.

It follows from the classical formula \eqref{E:derived_parabolic_1} and \eqref{E:derived_parabolic_2} that one can rewrite the solution \eqref{E:hautot_sum_7} in the form
	\begin{equation}	
		\label{E:induction-step-1}
		f(x) = \sum_{k=0}^{N}A_k e^{x^2/4}D_{\frac{e}{2}-k}(x)+A_{N+1} e^{x^2/4}D_{\frac{e}{2}-N-1}(x).
	\end{equation}
If follows from inductive hypothesis that the first summand in \eqref{E:induction-step-1} is already in the desired form
	\begin{equation}
		\sum_{k=0}^{N}A_k e^{x^2/4}D_{\frac{e}{2}-k}(x)=	r_1(x) e^{x^2/4} D_{\frac{e}{2}-N}(x)
		+ r_2(x) e^{x^2/4} D_{\frac{e}{2}-N}^\prime(x)
	\end{equation}
for some polynomials $r_1(x),\, r_2(x)$. We deduce easily from \eqref{A:bateman_1} that
	\[
	D_{\frac{e}{2}-N}^\prime(x)= \big(\frac{e}{2}-N\big)D_{\frac{e}{2}-N-1}(x)-\frac{x}{2} D_{\frac{e}{2}-N}(x)
		\]
	holds.
It follows from this formula and \eqref{E:derived_parabolic_1}, \eqref{E:derived_parabolic_2} that we can rewrite the above equation in the form
	\begin{equation}\label{E:induction-step-2}
		\begin{split}
		r_1(x)\, e^{x^2/4} & D_{\frac{e}{2}-N}(x)
		 + r_2(x)\, e^{x^2/4} D_{\frac{e}{2}-N}^\prime(x)\\
		&=-r_1(x)\, e^{x^2/4}\Big[\Big(\frac{e}{2}-N-1\Big) D_{\frac{e}{2}-N-2}-x D_{\frac{e}{2}-N-1}\Big]\\
		&\qquad + r_2(x)\, e^{x^2/4}
		\Big[\big(\frac{e}{2}-N\big) D_{\frac{e}{2}-N-1}-\frac{x}{2}\, D_{\frac{e}{2}-N}\Big]\\
		&=-r_1(x)\big(\frac{e}{2}-N-1\big)e^{x^2/4}D_{\frac{e}{2}-N-2}+xr_1(x)e^{x^2/4}D_{\frac{e}{2}-N-1}\\
		&\qquad+ \big(\frac{e}{2}-N\big) r_2(x)e^{x^2/4} D_{\frac{e}{2}-N-1}-\frac{x}{2} r_2(x)e^{x^2/4} D_{\frac{e}{2}-N}\\
		&=- r_1(x)\frac{d}{dx}\big( e^{x^2/4}D_{\frac{e}{2}-N-1}\big) +xr_1(x)e^{x^2/4}D_{\frac{e}{2}-N-1}\\
		&\qquad + \big(\frac{e}{2}-N\big) r_2(x)e^{x^2/4} D_{\frac{e}{2}-N-1}+\frac{x}{2} r_2(x)e^{x^2/2}\frac{d}{dx}\big(e^{-x^2/4}D_{\frac{e}{2}-N-1}\big).
		\end{split}
	\end{equation}
But since both
	\[
		\frac{d}{dx}\big( e^{x^2/4}D_{\frac{e}{2}-N-1}\big)
		\quad
		\textrm{and}
		\quad
		e^{x^2/2}\frac{d}{dx}\big(e^{-x^2/4}D_{\frac{e}{2}-N-1}\big)
		\]
are linear combinations of $e^{x^2/4} D_{\frac{e}{2}-N-1}$ and $e^{x^2/4} D_{\frac{e}{2}-N-1}^\prime$ with polynomial coefficients. Combining this fact, \eqref{E:induction-step-1}, \eqref{E:induction-step-2} that the \eqref{E:gauge_down} holds when either $e=\gamma-\alpha-2=2N$ or $c=\alpha+1=-N$ for some integer $N\ge 1$ holds.
\end{proof}

\section{Schlesinger transformations and apparent singularity:  completion of the proof of Theorem \ref{T:BHC-degenerate-3}}\label{S:completion}
\setcounter{equation}{0}

We recall that the parabolic cylinder equation \eqref{E:parabolic} corresponds to a parabolic type connection \eqref{E:parabolic-connection} that can be written in the form
	\begin{equation}\label{E:parabolic_connection_2}
		\frac{d\Phi}{dx}=\big(A^\prime x+ B^\prime\big) \Phi=\mathcal{P}\Phi,
		\qquad
			A^\prime=\frac12
			\begin{pmatrix}
				1 & 0\\
				0 & -1
			\end{pmatrix},
		\quad
			B^\prime
			=
			\begin{pmatrix}
				0 & r\\
				s & 0
			\end{pmatrix}
	\end{equation}
where $\nu+1=-rs$ (See e.g., \cite[\S1.5]{FIKN})
	\begin{equation}\label{E:parabolic_connection_soln}
		\Phi(x)=
		\begin{pmatrix}
			\Phi_{11} & \Phi_{12}\\
			\Phi_{21} & \Phi_{22}
		\end{pmatrix},
		\qquad
		\Phi_{11}(x)=D_\nu(x),
		\quad
		\Phi_{12}(x)=E_\nu(x)
	\end{equation}
and**\footnote{** We note that we have adopted $E_\nu(z)$ as defined in \eqref{E:parabolic_2nd_soln} as the second linearly independent solution to \eqref{E:parabolic} instead of adopting the $D_{-\nu-1}(ix)$ as the second linearly independent solution.}
	\begin{equation}\label{E:parabolic_connection_soln_2}
		\Phi_{21}(x)=\frac{1}{r}\big(D_\nu^\prime(x)-\frac{x}{2}D_{\nu}(x)\big),
	\quad
		\Phi_{22}(x)=\frac{1}{r}\big(E_\nu^\prime(x)-\frac{x}{2}E_{\nu}(x)\big).
	\end{equation}

The proof of Theorem \ref{T:BHC-degenerate-3} can be completed after we have established the following theorem.

\begin{theorem}\label{E:schlesinger}  Let $\alpha+1\in \mathbb{Z}$ in the biconfluent type connection \eqref{E:BHC}. Then there exists a Schlesinger transformation $\mathcal{S}$ that transforms the parabolic connection \eqref{E:parabolic_connection_2} to the  biconfluent type connection \eqref{E:BHC}.
\end{theorem}

\begin{proof} Let
	\begin{equation}
			\widetilde{\Phi}(x)=
				\begin{pmatrix}
					D_\nu(x) & E_\nu(x)\\
					D_{\nu}^\prime(x) & E_\nu^\prime(x).
				\end{pmatrix}
	\end{equation}
Then one can write
	\[
		\Phi(x)=C\widetilde{\Phi}(x)=
		\begin{pmatrix}
			1 & 0 \\
			\displaystyle\frac1r & \displaystyle-\frac{x}{2r}
		\end{pmatrix}
		\begin{pmatrix}
			D_\nu &  E_\nu(x)\\
			D_{\nu}^\prime(x) & E_\nu^\prime(x).
		\end{pmatrix}
	\]

The remaining step of the proof is to show the existence of a Schlesinger transformation from a parabolic type connection to a biconfluent type connection. The difference between our argument and that of the general theory of Schlesinger transformations is that Schlesinger transformations generally transform solutions amongst connections of the same monodromy/Stokes' data, while our Schlesinger transformations to be constructed transform between connections of different monodromy/Stokes' data in the sense that one connection has fewer singularities than the other.

We first assume that $2\theta_0=\alpha+1=-N\le -1$. That is, $\alpha\in\big\{-2,\, -3,\, -4,\cdots
\big\}$.
Without loss of generality, we may assume $z=0$, $y=0$ and $z/y=\lambda$ in \eqref{E:BHC}. That is, we consider the member of isomonodromy family of the BHC \eqref{E:BHC} at the exceptional point $z=0$, $y=0$,  $z/y=\lambda$. Since all the isomonodromy integral curves in the $yz-$plane pass through the exceptional line of the blow up $(z,\, z/y)\mapsto (z,\, y)$ (see \cite{poberezhny}), no generality is lost. Thus one derives the \eqref{E:BHE} as stated. Since $\alpha+1=-N$ and $N\ge 1$ here, so Theorem \eqref{T:gauge} case 1 indicates that we can find two linearly independent solutions \eqref{E:gauge_up} and \eqref{E:gauge_up_b} to the BHE \eqref{E:BHE}.

It follows from Theorem \ref{T:gauge} (1) that for a suitable $j$ that both $f$ and $g$ are given by \eqref{E:gauge_down} and \eqref{E:gauge_down_b}. Moreover, we deduce that
\begin{equation}\label{E:gauge_down_II}
		f_j^\prime(x)=q_{0,\, j}(x)\, e^{x^2/4} D_{\frac{e}{2}-N}(x)+ q_{1,\, j}(x)\, e^{x^2/4}\big(D_{\frac{e}{2}-N}(x)\big)^\prime,
		\quad 0\le j\le N
	\end{equation}
and
	\begin{equation}\label{E:gauge_down__IIb}
		g_j^\prime(x)=q_{0,\, j}(x)\, e^{x^2/4} E_{\frac{e}{2}-N}(x)+ q_{1,\, j}(x)\, e^{x^2/4}\big(E_{\frac{e}{2}-N}(x)\big)^\prime,
		\quad 0\le j\le N
	\end{equation}
for the same polynomials $q_{0,\, j}(x)$ and $q_{1,\, j}(x)$. But we know that
		\[
			\Psi(x)=
			\begin{pmatrix}
				f & g \\
				\displaystyle\frac{1}{a_{12}}(f^\prime-a_{11}f) & \displaystyle\frac{1}{a_{22}} (g^\prime-a_{22}g)
			\end{pmatrix}
		\]
satisfies $\Psi^\prime(x)=\mathcal{A} \Psi(x)$. So we can find a matrix $Q(x)$ such that
		\[
			\Psi (x)=Q(x)\, \widetilde{\Phi}(x).
		\]
Hence
		\[
 			\Psi (x)=Q(x) \widetilde{\Phi}(x)= Q(x) (C^{-1}(x)\Phi(x))=
			(Q(x)C^{-1}(x))\, \Phi(x)
			=R(x)\, \Phi(x),
		\]
where $R(x)=Q(x)C^{-1}(x)$, and the $\Psi(x)$ satisfies the \eqref{E:BHC}. Substituting this $\Psi(x)=R(x)\, \Phi(x)$ and remembering that $\Phi(x)$ satisfies the parabolic connection $\Phi^\prime(x)=\mathcal{P}\Phi(x)$, yields the equation
	\[
		\Psi^\prime=\big(R\mathcal{P}R^{-1}+R_x\, R^{-1}\big) \Psi.
	\]
Hence
	\[
		\mathcal{A}=\big(R\mathcal{P}+R_x\, \big)R^{-1}
	\]
which is nothing but the standard formula that appears in Schlesinger transformation between the \eqref{E:parabolic-connection} and \eqref{E:BHC}. See e.g., \cite[(2.19)]{mugan-fokas}. This completes the proof when $2\theta_0=1+\alpha=\big\{-1,\, -2,\, -3,\cdots\big\}$.

We now suppose that $2\theta_0=\big\{1,\, 2,\, 3,\cdots\big\}$. But since Theorem \ref{T:BHC_sym} shows that the BHC \eqref{E:BHC} is invariant under the transformation $\theta_0\mapsto -\theta_0$ \eqref{E:theta_0}, so we could relegate this case to the above consideration. It remains to consider that case when $2\theta_0=0$. But this means that the matrix $C$ in \eqref{E:BHC-matrices} becomes trivial so that the BHC \eqref{E:BHC} immediately reduces to a parabolic type connection \eqref{E:parabolic-connection}. This completes the proof.

\end{proof}


\section{Series solutions of parabolic cylinder functions}\label{S:Sums}
 \setcounter{equation}{0}

  On the other hand,
 The {BHE} (\ref{E:BHE}) can be rewritten as
 $$
 y''+(-2z+b+\frac{c}{z})y'+(\frac{d}{z}+e)y=0.
 $$
 Let $\d P(z)=-2z+b+\frac{c}{z}$. Through the standard substitution,
 $$
 y=Y \, \exp({-\frac{1}{2}\int P})=Y(z)\,
 z^{\frac{-c}{2}}\, \exp({\frac{1}{2}z^2-\frac{1}{2}b z}),
 $$
 the above equation becomes
 \begin{eqnarray*}
 0 &=& Y''-\left(
 z^2-b z+(\frac{b^2}{4}-1-c-e)+(\frac{b c}{2}-d)\frac{1}{z}+(\frac{c^2}{4}-\frac{c}{2})\frac{1}{z^2}\right)Y\\
  &=& Y''-\left(
  \frac{1}{2}(\frac{2z-b}{\sqrt{2}})^2-(1+c+e)+O(\frac{1}{z})\right)Y.
  \end{eqnarray*}
 This suggests that the equation is asymptotic to
 the parabolic cylinder equation \eqref{E:parabolic}
 \begin{equation*}
  D_\nu''(x)=(\frac{1}{4}x^2-\nu-\frac{1}{2})D_\nu(x).
  \end{equation*}
We note that two identities \eqref{E:derived_parabolic_1} and \eqref{E:derived_parabolic_2} for which any solution $D_\nu(x)$ of \eqref{E:parabolic} satisfy are refereed to in Appendix B.

 Let $y(z)=f(x)$, where $\d x=\frac{b-2z}{\sqrt{2}}$
 or equivalently $\d z=\frac{b-\sqrt{2}x}{2}$.
 Then (\ref{E:BHE}) becomes
 $$
 ({b-\sqrt{2}x})f''+(\sqrt{2}x^2-b
 x-\sqrt{2}c)f'+(\frac{-e}{\sqrt{2}}x+d+\frac{b e}{2})f=0.
 $$

We extend Hautot's work \cite{hautot_1969, hautot_1971} by showing the formal expansion solution
\begin{equation}\label{E:infinite_sum_2}
		y(x)=e^{x^2/4}\sum_{k=0}^\infty A_k D_{\frac{e}{2}-k}(x)
		=e^{x^2/4}\sum_{k=0}^\infty A_k D_{(\theta_\infty-\theta_0-1)-k}(x)
,\quad x=(b-2z)/\sqrt{2}
	\end{equation}
of \eqref{E:BHE} mentioned in \eqref{E:infinite_sum} actually converges in any compact set in a half-plane uniformly. We achieve this by applying a recent result of Wong and Li \cite{wong} on linear second order difference equations.

\begin{theorem} \label{T:Wong-Li}
Consider the second order linear difference equation
 \begin{equation}
 y(n+2)+n^p a(n)y(n+1)+n^q b(n)y(n)=0,\label{eq2.5}
 \end{equation}
 where $p$ and $q$ are integers, and $a(n)$ and $b(n)$ have power series
 expansions of the form
	\begin{equation}
 a(n)=\sum_{s=0}^\infty \frac{a_s}{n^s}\qquad  \textmd{and} \qquad
 b(n)=\sum_{s=0}^\infty \frac{b_s}{n^s},
	\end{equation}
 for large values of $n$, $a_0\neq 0$, $b_0\neq 0$. When $2p-q=-1,$ two
 formal series solutions of (\ref{eq2.5}) are of the form
	\begin{equation}\label{E:Wong-Li-asymptotic}
 y(n)=[(n-2)!]^{q/2}\rho^n e^{\gamma \sqrt{n}} n^\alpha
 \sum_{s=0}^{\infty} \frac{c_s}{n^{s/2}},
	\end{equation}
 where
	\begin{equation}
 \rho^2=-b_0,\quad \gamma=-\frac{a_0}{\rho}\quad \textmd{and}\quad
 \alpha=\frac{b_1}{2b_0}+\frac{q}{4}.
	\end{equation}
 \end{theorem}

The estimate \eqref{E:Wong-Li-asymptotic} gives an accurate asymptotic behaviour of the coefficients $A_k$ of \eqref{E:infinite_sum_2}, which satisfies the three-term recursion \eqref{E:3-recursion}.

\begin{theorem}\label{T:half-expansion} Suppose $\Re b=Re(-\beta)<0$. Then the infinite parabolic sum  solution \eqref{E:infinite_sum_2} of the biconfluent Heun equation\eqref{E:BHE} converges absolutely in any compact subset in $\Re z<\Re b < 0$.
\end{theorem}

\begin{proof}
Now from Theorem \ref{T:formal-series} that the recurrence relation can be written
 as
 $$
 A_{n+2}-\frac{d+b(n+1)}{\sqrt{2}(n+2)}A_{n+1}+\frac{(\frac{e}{2}-n)(n+c)}{n+2}A_n=0.
 $$
 Here
 $$
 -\frac{d+b(n+1)}{\sqrt{2}(n+2)} =
 \frac{-b}{\sqrt{2}}\left(1+\frac{d-b}{bn}+O(\frac{1}{n^2})\right),
 $$
 and
 $$
 \frac{(\frac{e}{2}-n)(n+c)}{n+2}
 =n\left(-1+\frac{2+\frac{e}{2}-c}{n}+O(\frac{1}{n^2})\right).
 $$
 So $p=0$, $q=1$, $\d a_0=\frac{-b}{\sqrt{2}}$, $b_0=-1$, and
 $\d b_1=2+\frac{e}{2}-c$ in \eqref{E:3-term}. Thus
 $$
 \rho=\pm 1,\quad \ga=\mp \frac{b}{\sqrt{2}}\quad
 \al=\frac{c}{2}-\frac{e+3}{4}.
 $$
 Hence, by Theorem \ref{T:Wong-Li}, as $n\to\infty$,
 $$
 A_n\sim [(n-2)!]^{1/2}(\pm 1)^n \rme^{\mp b \sqrt{n/2}}\,
 n^{\frac{c}{2}-\frac{e}{4}-\frac{3}{4}}.
 $$
 By Stirling's formula, we have
	\begin{equation}\label{E:coeff-asymptote}
 \log A_n \sim \frac{2n-3}{4}\log (n-2)-\frac{n}{2}\pm
 b\sqrt{\frac{n}{2}}+(\frac{c}{2}-\frac{e}{4}-\frac{3}{4})\log n+C.
	\end{equation}

 We deduce from \eqref{E:parabolic-asymptote} that, as $n\to\infty$,
 $$
 D_{\frac{e}{2}-n}(x)\sim \frac{1}{\sqrt{2}} \exp\left(
 (\frac{e}{4}-\frac{n}{2})\log(n-\frac{e}{2})+(\frac{n}{2}-\frac{e}{4})-x(n-\frac{e}{2})^{1/2}\right).
 $$
 Hence
 \begin{equation}
 \log D_{\frac{e}{2}-n}(x)\sim \frac{e}{4}\log n-\frac{n}{2}\log n
 +\frac{n}{2}-x\sqrt{n}+C ,\label{eq2.7}
 \end{equation}
 where $C$ is a constant.

 Combining it with \eqref{E:coeff-asymptote}, we have a simple asymptotic
 expression.
 \begin{equation}
 \log A_n D_{\frac{e}{2}-n}(x) \sim (\pm
 \frac{b}{\sqrt{2}}-x)\sqrt{n}+(\frac{c}{2}-\frac{3}{2})\log n +C,
 \label{eq2.8}
 \end{equation}
 where $x=\Re z$.  Therefore the series \eqref{E:infinite_sum_2} converges absolutely if
 $$
 \Re (\frac{b}{\sqrt{2}}-x),\ \Re(\frac{-b}{\sqrt{2}}-x)<0,
 $$
 or equivalently,
 $$
 \Re z,\ \Re (z-2b) < 0.
 $$
 For in this case, there exists some $\ep>0$ such that
 $$
 \log A_n D_{\frac{e}{2}-n}(x) \sim -\ep\sqrt{n}.
 $$
 So
 $$
 |A_n D_{\frac{e}{2}-n}(x)| \sim \rme^{-\ep\sqrt{n}}<\frac{1}{n^2}
 $$
 where $n$ is large enough. Hence the series is absolutely
 convergent.
\end{proof}

We next show that it is possible to construct an entire function solution to \eqref{E:BHE}
from the \eqref{E:infinite_sum_2}. We resort to consider solutions $\phi_5$ and $\phi_4$ from the symmetry group as described in Proposition \ref{T:BHE-symmetry}.

\subsection{Asymptotic behaviour of $A_kD_{\frac{e}{2}-k}$ from $\phi_5(z)$}

We recall the relationship between the two sets of parameters $(\alpha,\, \beta,\, \gamma,\, \delta)$ and $(a,\, b,\, c,\, d)$ of the equations \eqref{E:BHEC} and \eqref{E:BHE} respectively, are given by
 $$
 (\al,\be,\ga,\de)=(c-1,-b,e+c+1,bc-2d),
 $$
 {so that}
 $$
 {(b,c,d,e)=(-\be,\al+1,-\frac{1}{2}(\de+(1+\al)\be),\ga-\al-2).}
 $$
 We want to study the transformation $\phi_5$ in Proposition \ref{T:BHE-symmetry} which maps $(b,d,e)$ to
 $$
 (\widetilde{b},\widetilde{d},\widetilde{e})= (i b, i(b c-d),-e-2c-2).
 $$
 Thus the formal series in \eqref{E:BHE} becomes
 \begin{eqnarray*}
 \phi_4(z) &=& \exp(\be z+z^2) \exp \frac{(2i z-i b)^2}{8}
 \sum_{n=0}^\infty \widetilde{A_n} D_{-\frac{e}{2}-c-1-n} (\frac{i b-2 i
 z}{\sqrt{2}})\\
  &=& \exp (\frac{1}{2}(z-\frac{b}{2})^2) \sum_{n=0}^\infty \widetilde{A_n} D_{-\frac{e}{2}-c-1-n}
   (\frac{i (b-2 z)}{\sqrt{2}}),
  \end{eqnarray*}
  while the recurrence relation becomes
  $$
  -\sqrt{2}(n+2) \widetilde{A}_{n+2}+i(b(n+1)+b
  c-d)\widetilde{A}_{n+1}+\sqrt{2}(n+\frac{e}{2}+c+1)(n+c)\widetilde{A}_n=0,
  $$
  or
  \begin{equation}
  \widetilde{A}_{n+2}-\frac{i(b n+b+b
  c-d)}{\sqrt{2}(n+2)}\widetilde{A}_{n+1}+\frac{(n+\frac{e}{2}+c+1)(n+c)}{n+2}
  \widetilde{A}_n=0. \label{eq2.11}
  \end{equation}
  We also let $\widetilde{A}_0=1$, so $\d
  \widetilde{A}_1=\frac{i}{\sqrt{2}(b c-d)}$.
  Hence $\d a_0=\frac{-i b}{\sqrt{2}}$, $b_0=-1$, $\d
  b_1=1-2c-\frac{e}{2}$, which implies
  $$
  \rho_0=\pm 1,\quad \ga_0=\pm\frac{i b}{\sqrt{2}},\quad \al_0=
  c+\frac{e}{4}-\frac{1}{4}.
  $$
  So by Theorem \ref{T:Wong-Li}, we have
  $$
  \widetilde{A}_n\sim [(n-2)!]^{1/2}(\pm 1)^n \rme^{\pm i
  b \sqrt{\frac{n}{2}}} n^{c+\frac{e}{4}-\frac{1}{4}}.
  $$
  That means,
  $$
  \log \widetilde{A}_n \sim \frac{2n-3}{4}\log(n-2)-\frac{n}{2}\pm i b
  \sqrt{\frac{n}{2}}+(c+\frac{e}{4}-\frac{1}{4})\log n + C.
  $$
  By the same transformation, (\ref{eq2.7}) becomes
  \begin{eqnarray*}
  \log D_{-\frac{e}{2}-c-1-n}(\frac{i(b-2z)}{\sqrt{2}} &\sim&
  (\frac{-e}{4}-\frac{c+1+n}{2})\log (n+1+c+\frac{e}{2})+
  \frac{n+c+1+\frac{e}{2}}{2}\\
   &&\qquad -\sqrt{n+c+1+\frac{e}{2}}\,
  \frac{i(b-2z)}{\sqrt{2}}+C.
  \end{eqnarray*}
  Combining the asymptotics,
  \begin{equation}
   \widetilde{A}_n
  D_{-\frac{e}{2}-c-1-n}(\frac{i(b-2z)}{\sqrt{2}})
  \sim \frac{i}{\sqrt{2}}(\pm
  b-b+2z)\sqrt{n}+(\frac{c}{2}-\frac{3}{4})\log n +C.\label{eq2.9}
  \end{equation}
  The series converges absolutely if
 $\Im (\pm b-b+2z)>0$, or $\Im (\pm b-b+2z)=0$ but $\Re
 (\frac{c}{2}-\frac{3}{4})<-1$. This is equivalent to say that
 $$
 \Im z>0\ \mbox{and}\ \Im z>\frac{1}{2}\Im b,\quad \mbox{or}\quad \mbox{one equality holds while } \Re
 c<\frac{-1}{2}
 $$
 while one of them vanishes.

\subsection{Asymptotic behaviour of $A_kD_{\frac{e}{2}-k}$ from $\phi_4(z)$}

On the other hand, we consider the transformation $\phi_4(z)$ from Proposition \ref{T:BHE-symmetry} which maps the parameters $(b,d,e)$ to $(-i b, i(b c-d),-e-2c-2)$. Hence
 the series expansion becomes
 $$
 \phi_3(z)= \exp(\frac{1}{2}(z-\frac{b}{2})^2)\sum_{n=0}^\infty
 \widetilde{B}_n D_{-\frac{e}{2}-c-1-n}(\frac{i}{\sqrt{2}}(2z-b)),
 $$
 where
 \begin{equation}
 \widetilde{B}_{n+2}+\frac{i(b n+b+b c-d)}{\sqrt{2}(n+2)}
 \widetilde{B}_{n+1}-\frac{(n+c)(n+\frac{e}{2}+c+1)}{n+2}\widetilde{B}_n=0.
 \label{eq2.12}
 \end{equation}
 Letting $\widetilde{B}_0=1$, we also have $\d
  \widetilde{B}_1=\frac{i}{\sqrt{2}(b c-d)}$.
 So $\d a_0=\frac{ib}{\sqrt{2}}$, $b_0=-1$, $b_1=1-2c-\frac{e}{2}$,
 which implies that
 $$
 \rho=\pm 1,\quad \ga_0=\mp \frac{ib}{\sqrt{2}},\quad
 \al_0=c+\frac{e}{4}-\frac{1}{4}.
 $$
 Similar as above, we obtain
 $$
 \log \widetilde{B}_n \sim \frac{2n-3}{4}\log(n-2)-\frac{n}{2}\mp
 ib\sqrt{\frac{n}{2}}+(c+\frac{e-1}{4})\log n +C.
 $$
 Also,
 $$
 \log D_{-\frac{e}{2}-c-1-n}(\frac{i(2z-b)}{\sqrt{2}} \sim
  (\frac{-e}{4}-\frac{c+1+n}{2})\log n+
  \frac{n}{2}-
  {i(2z-b)}{\sqrt{\frac{n}{2}}}+C.
  $$
  Therefore, we have
  $$
   \widetilde{B}_n
  D_{-\frac{e}{2}-c-1-n}(\frac{i(2z-b)}{\sqrt{2}})
  \sim \frac{-i}{\sqrt{2}}(\pm
  b-b+2z)\sqrt{n}+(\frac{c}{2}-\frac{3}{4})\log n +C.
  $$
  So the series converge absolutely if
  $$
  \Im (2z-b\pm b)<0,\quad \mbox{or}\quad \Im(2z-b\pm b)=0\mbox{ but } \Re
  c<-\frac{1}{2}.
  $$
  The above first condition is equivalent to $\Im z<0$ and $\Im
  z<\Im b$.

We are ready to state the next main result.

\begin{theorem}\label{T:full-expansion}
If $b\in \bfR$ and $\d \Re c<-\frac{1}{2}$, then there exists an
  entire solution $y$ of (\ref{E:BHE}) expressible in terms of an absolutely convergent
  series of parabolic cylinder functions:
  $$
  y(z)=\left\{
  \begin{array}{ll}
  C_0 \exp(\frac{1}{2}(z-\frac{b}{2})^2)\sum \widetilde{A}_n  D_{-\frac{e}{2}-c-1-n}\displaystyle(\frac{i(b-2z)}{\sqrt{2}})
  :=C_0 \Phi(z) & \Im
  z\geq 0\\
  \exp(\frac{1}{2}(z-\frac{b}{2})^2)\sum \widetilde{B}_n  D_{-\frac{e}{2}-c-1-n} \displaystyle(\frac{i(2z-b)}{\sqrt{2}})
  := \Psi(z) & \Im
  z\leq 0
  \end{array} \ ,\right.
  $$
  where $\widetilde{A}_n$ is given by (\ref{eq2.11})  while
  $\widetilde{B}_n$ is given by (\ref{eq2.12}) with
  $\widetilde{A}_0=\widetilde{B}_0=1$ and $\widetilde{A}_1=\widetilde{B}_1=\frac{i}{\sqrt{2}}(b
  c-d)$. Also
  $C_0={\Phi}(\frac{b}{2})/{\Psi}(\frac{b}{2})$.
  \end{theorem}

\begin{proof}
First, by \cite[Prop 3.1]{ronv} and \cite[Prop
  3.2]{ronv}, there is an entire solution $N_1(z)$ of (\ref{E:BHE})
  as well as another linearly independent solution  $N_2(z)$ which
  has a (branch cut) singular point at $z=0$. Now $\d {\Phi}$ is
  a solution of (\ref{E:BHE}) when $\Im z>0$. So we have
  $$
  {\Phi}(z)=C_1 N_1(z)+C_2N_2(z),
  $$
  where $C_1$ and $C_2$ are constants. We claim that ${\Phi}$
  is continuous on $\{ \Im z\leq 0 \}$, so that by taking limit
  $z\to 0$, we have $C_2=0$ since ${\Phi}(0)$ is still finite.
  Thus ${\Phi}=C_1 N_1$. Similarly we also have
  ${\Psi}=\widetilde{C_1} N_1$. Therefore if we take
  $$
  y(z)=\left\{
  \begin{array}{ll}
  C_0 {\Phi}(z) & \Im
  z\geq 0\\
   {\Psi}(z) & \Im
  z\leq 0
  \end{array} \ ,\right.
  $$
  where
  $$
  C_0=\frac{\widetilde{C_1}}{C_1}=\frac{{\Phi}(\frac{b}{2})}{{\Psi}(\frac{b}{2})},
  $$
  then $y=\widetilde{C_1}N$ is an entire solution of (\ref{E:BHE}).

  So it suffices to show that $\Phi$ is continuous on $\{
  \Im z\geq 0\}$. With our restriction on $b$ and $c$, the equation
  (\ref{eq2.9}) with the discussion after it infers that the series
  $\d \sum_{n=0}^\infty \widetilde{A_n} D_{-\frac{e}{2}-c-1-n} (\frac{i(b-2z}{\sqrt{2}})$
  is absolutely and uniformly convergent on the
 half-plane $\{\Im z\geq 0\}$.  Indeed, we know that
 $$
 |\widetilde{A_n} D_{-\frac{e}{2}-c-1-n} (\frac{i b-2 i
 z}{\sqrt{2}})|\leq \rme^C n^{\frac{c}{2}-\frac{3}{4}},
 $$
 where $\d \Re(\frac{2c-3}{4})<-1$, or $\d \Re c<-\frac{1}{2}$.  Since $\sum n^{\frac{2c-3}{4}}$ is
 convergent, by Weierstrass $M$-test, we conclude that the above
 mentioned series is absolutely and uniformly convergent on the set
 $\{ \Im z\geq 0\}$.  This implies directly that the resulted function ${\Phi}$
 is continuous on the domain $\{ \Im z\geq 0\}$. The case of
 ${\Psi}$ is similar.
 \end{proof}

 \section{Concluding remarks} 
 \setcounter{equation}{0} 
 
 We have made a systematic study of the degenerate  monodromy and Stokes' data of the BHE \eqref{E:BHEC} and BHC \eqref{E:BHC} in this paper.
 In particular, we have
  demonstrated that there is a new complete correspondence between the parameter space of the BHC and that of thier isomondromy deformation counterparts, namely, the Painlev\'e IV equation. More precisely, we have identified if the parameters $\eta,\, \xi$ coincide with the affine Weyl group $\widetilde{A}_2^{(1)}$ discovered by Okamoto, i.e., $\eta$ and $\xi$ in \eqref{E:P4} satisfy $\eta=-2(2n+1+\varepsilon \xi)^2$ and/or $\eta=-2n^2,\ 	n\in\mathbb{Z}$ where $\varepsilon=\pm 1$, then the monodromy/Stokes' data $\theta_0$ and $\theta_\infty$, defined by Jimbo and Miwa \cite{JM_1981b}, satisfy $\theta_0\pm\theta_\infty\in\mathbb{Z}$ and/or $2\theta_0\in \mathbb{Z}$.  The converse of the above statement also holds. The former relation corresponds to the BHC with appropriate choices of accessory parameters, having Liouvillian solutions in the language of differential Galois theory, while the latter corresponds to $x=0$ reduces to an apparent singularity. Moreover, we have demonstrated that the BHC \eqref{E:BHC} can be transformed from a parabolic connection \eqref{E:parabolic-connection} in this latter case.  We have also derived explicit solutions for the BHE \eqref{E:BHE_JM}  in the forms of \eqref{E:hautot_sum_0} after choosing appropriate accessory parameters $\lambda$ that are counterparts of special (function/rational) solutions of $\mathrm{P_{IV}}$ \textit{except at three cases/integers} $n$. In particular, both can be written as finite combinations of parabolic cylinder functions. The BHE \eqref{E:BHE_JM} admits eigen-solution of the form
  	\[
		y(x)=e^{x^2/4}\sum_{k=0}^N A_k D_{(\theta_\infty-\theta_0-1)-k}(x),\quad x=(b-2z)/\sqrt{2}.
	\]
We have developed a theory of invariant subspaces spanned by $e^{x^2/4} D_\nu(x)$ and $e^{x^2/4} E_\nu(x)$ and their derivatives so that these explicit solutions \eqref{E:hautot_sum_0} derived above are eigen-solutions.

We have extended the finite sum above to an infinite sum hence obtaining more general solution
	\begin{equation}\label{E:finite-sum}
		y(x)=e^{x^2/4}\sum_{k=0}^\infty A_k D_{(\theta_\infty-\theta_0-1)-k}(x),\quad x=(b-2z)/\sqrt{2}.
	\end{equation}
to the BHE \eqref{E:BHE_JM} that is converging in an half-plane. We need to apply the latest asymptotic result of solutions to second order difference equations of Wong and Li \cite{wong} in order to compute for the asymptotics of the coefficients $A_k$.

Although we have demonstrated that the BHC and $\mathrm{P_{IV}}$ share the same parameter space for their degenerations, the actual algebraic structure apart from symmetry of BHC is a problem for future consideration. Another problem is to derive closed form expressions for the $A_k$ involved in the above finite sums. It is now clear that this problem is equivalent to finding closed form expressions for the corresponding Schlesinger transformations $\theta_0\mapsto \theta_0+n$ and $\theta_\infty\mapsto \theta_\infty+n$. Finally, we mentioned that we have also found orthogonality relations exist amongst the eigen-solutions \eqref{E:finite-sum}. We shall address to the above issues in the near future.

\section*{Appendix A: BHC and BHE}  
 \setcounter{equation}{0} 
 
 The following formula records a relation between a differential equation in system form to its scalar form.
\begin{proposition}\label{P:system_scalar}
Let $A=(a_{ij})_{1\leq i,j\leq 2}$ be a matrix-valued holomorphic function. If $Y=(y_1,y_2)^T$ is a vector-valued function satisfying $Y'=AY$. Then,
	\begin{equation}\label{E:conversion}
y''_1+[-a_{11}-a_{22}-\dfrac{a'_{12}}{a_{12}}]y'_1+[a_{11}a_{22}-a_{21}a_{12}-a_{12}(\dfrac{a_{11}}{a_{12}})']y_1=0.
	\end{equation}
Moreover,
	\begin{equation}\label{E:conversion_2}
		y_2=\frac{1}{a_{12}}\big( y_{1}^\prime-a_{11}y_{1}\big).
	\end{equation}		
\end{proposition}
\begin{proof}
Direct calculation.
\end{proof}

\section*{Appendix B: Parabolic cylinder functions} 

 We list some formulae for parabolic cylinder functions in the convention as adopted by Whittaker. Our primary references are \cite[Chap. VIII]{erdelyi} and \cite[Chap. 19]{AS}. Let*\footnote{*See either \cite[\S8.2 (4)]{erdelyi} or \cite[19.3.1]{AS}}
	\begin{equation*}
		\begin{split}
		D_\nu(x) &= 2^{\nu/2} e^{-x^2/4}\Big[
		\cos\big(\frac{\pi \nu}{2}\big) \Gamma\big( \frac{\nu+1}{2}\big)
		{}_1\mathrm{F}_1\big(-\frac{\nu}{2},\, \frac12;\, \frac{x^2}{2}\big)\\
		&
		\qquad +
		\sin\big(\frac{\nu+1}{2}\big)\pi \Gamma\big(\frac{\nu}{2}+1\big)
		x\, {}_1\mathrm{F}_1\big(-\frac{\nu}{2}+\frac12,\, \frac32;\, \frac{x^2}{2}\big)\Big]
		\end{split}
	\end{equation*}
We have \cite[p. 119]{erdelyi}:
	\begin{equation}\label{A:bateman_1}
		\frac{d^k}{dx^k}\big[ e^{x^2/4} D_\nu(x)\big]=(-1)^k(-\nu)_k\, e^{x^2/4} D_{\nu-k}(x),
	\end{equation}
and
	\begin{equation}\label{A:bateman_2}
		\frac{d^k}{dx^k}\big[ e^{-x^2/4} D_\nu(x)\big]=(-1)^k e^{-x^2/4} D_{\nu+k}(x),
	\end{equation}
$k=1,\, 2,\, 3,\, \cdots$. We can derive the following recursion formulae from the above two differential identities:
\begin{eqnarray}
  2 D_\nu'(x)&=&-D_{\nu+1}(x)+\nu D_{\nu}(x),\label{E:derived_parabolic_1}\\
  x D_\nu(x) &=& D_{\nu+1}(x)+\nu D_{\nu-1}(x)\label{E:derived_parabolic_2}.
  \end{eqnarray}

 It is known from \cite[p. 123 (5)]{erdelyi} that for $\d |\arg(-\nu)|\leq
 \frac{\pi}{2}$,
	\begin{equation}\label{E:parabolic-asymptote}
 D_\nu(x)=\frac{1}{\sqrt{2}}\exp\big[
 \frac{\nu}{2}\log(-\nu)-\frac{\nu}{2}-\sqrt{-\nu}x\big]\,
 \big(1+O(|\nu|^{-1/2})\big),
	\end{equation}
holds.

Let
	\begin{equation}\label{E:parabolic_2nd_soln}
		V(x;\, -\nu-\frac12 ):= \frac{1}{\sin \pi\nu} \Big[ \sin \pi\big(\nu+\frac12\big)\, D_{-\nu}(x)-D_{-\nu}(-x)\Big]
	\end{equation}
where $E_{-a-\frac12}(x):=\Gamma\big(\frac12-a\big) V(a,\, x)$ as in \cite[19.3.8]{AS} serves as a second linearly independent solution to \eqref{E:parabolic} that also satisfies the formulae \eqref{A:bateman_1}, \eqref{A:bateman_2}, \eqref{E:derived_parabolic_1} and  \eqref{E:derived_parabolic_1}. See e.g. \cite[19.6]{AS}.

\section*{\bf Acknowledgements}
This work was supported in part by the GRF (No.: 16300814) of the Hong Kong Research Grant Council, Postdoctoral Development Fund of HKUST,
 National Natural Science
Foundation of China under Grant No.  11871336, as well as
Ministry of Science and Technology, Taiwan under Grant No. MOST104-2115-M-110-008. Moreover, the authors would like to express their thanks for an inspiring conversation with Ph. Boalch, and valuable conversations with Avery Ching and Chiu-Yin Tsang and for their interests to our paper. 

\end{document}